\documentclass[11 pt,reqno]{amsart}
\usepackage{amssymb,amsmath,amstext,amsgen,amsfonts,amsthm,verbatim,hyperref,fullpage,enumitem,graphicx}
\usepackage[usenames,dvipsnames]{color}
\usepackage{amsrefs}

\newcommand{\N}{\mathbb{N}}
\newcommand{\R}{\mathbb{R}}

\newcommand{\Z}{\mathbb{Z}}

\renewcommand{\H}{\mathbb{H}}

\newcommand{\g}{\mathfrak{g}}

\renewcommand{\epsilon}{\varepsilon}
\renewcommand{\phi}{\varphi}
\renewcommand{\hat}{\widehat}
\renewcommand{\tilde}{\widetilde}

\newcommand{\cB}{\mathcal{B}}

\newcommand{\cF}{\mathcal{F}}
\newcommand{\cG}{\mathcal{G}}
\newcommand{\Hd}{\mathcal{H}}

\newcommand{\bT}{\mathbb{T}}

\DeclareMathOperator{\diam}{diam}

\DeclareMathOperator{\Center}{Center}

\newtheorem{theorem}{Theorem}[section]
\newtheorem{proposition}[theorem]{Proposition}
\newtheorem{lemma}[theorem]{Lemma}
\newtheorem{corollary}[theorem]{Corollary}

\theoremstyle{remark}

\begin{document}

\title{Stratified $\beta$-numbers and traveling salesman in Carnot groups}
\author{Sean Li}
\date{\today}
\subjclass[2010]{Primary 28A75, 53C17}
\keywords{Heisenberg group, Traveling Salesman Theorem, Jones $\beta$ numbers, curvature}
\address{Department of Mathematics, University of Connecticut, Storrs, CT 06269}
\email{sean.li@uconn.edu}

\begin{abstract}
  We introduce a modified version of P. Jones's $\beta$-numbers for Carnot groups which we call {\it stratified $\beta$-numbers}.  We show that an analogue of Jones's traveling salesman theorem on 1-rectifiability of sets holds for any Carnot group if we replace previous notions of $\beta$-numbers in Carnot groups with stratified $\beta$-numbers.  As we generalize both directions of the traveling salesman theorem, we get a characterization of subsets of Carnot groups that lie on finite length rectifiable curves.  Our proof expands upon previous analysis of the Hebisch-Sikora norm for Carnot groups.  In particular, we find new estimates on the drift between almost parallel line segments that take advantage of the stratified $\beta$'s and also develop a Taylor expansion technique of the norm.  We also give an example of a Carnot group for which a traveling salesman theorem based on the unmodified $\beta$-numbers must exhibit a gap between the necessary and sufficient directions.
\end{abstract}

\maketitle

\section{Introduction}

The analyst's traveling salesman problem asks under what conditions does a subset $E$ of a metric space $X$ lies on a finite length rectifiable curve.  Whether there is a solution for any subset $E$ depends heavily on the geometry of the metric space $X$.  In 1990, Peter Jones gave a solution to the problem in $\R^2$ \cite{Jones-TSP} via the introduction of the so-called $\beta$-numbers.

Given an arbitrary subset $E \subset \R^2$ and a ball $B(x,r)$, we define
\begin{align}
  \beta_E(x,r) := \inf_L \sup_{z \in E \cap B(x,r)} \frac{d(z,L)}{r} \label{e:og-beta}
\end{align}
Here, the infimum is taken over affine lines $L \subset \R^2$.  We see that $\beta_E(x,r)$ is a number in $[0,1]$ that can be thought of as the minimal (rescaled) radius tube that contains $E$ in $B(x,r)$ and so measures how close $E$ lies to some affine line.

Given the notion of $\beta$-numbers, Jones proved the following theorem, which is now known as the traveling salesman theorem.

\begin{theorem}[\cite{Jones-TSP}] \label{th:jones}
  For a set in $\R^2$, we define
  \begin{align*}
    \beta(E) := \int_0^\infty \int_{\R^2} \beta_E(x,r)^2 ~dx \frac{dr}{r^2}.
  \end{align*}
  Then $E \subset \R^2$ lies on a finite length rectifiable curve if and only if $\diam(E) + \beta(E)$ is finite.  More precisely, there exists a universal constant $C > 1$ so that:
  \begin{enumerate}[label=(\arabic*)]
    \item If $\Gamma$ is any curve containing $E$ then $\beta(E) + \diam(E) \leq C\Hd^1(\Gamma)$,
    \item If $\beta(E) +\diam(E) < \infty$, then there exists a curve $\Gamma \supset E$ for which $\Hd^1(\Gamma) \leq C (\beta(E) + \diam(E))$.
  \end{enumerate}
\end{theorem}

The original theorem in Jones's paper has $\beta$ expressed in terms of a sum over dyadic cubes of $\R^2$ (and $\beta$ is also expressed in terms of cubes), but it is well known that the sum over cubes is equivalent to the integral over balls up to absolute multiplicative constants.  It is also important to note that the 2 in the exponent of $\beta^2$ comes from the power type of the modulus of convexity of $\R^n$ whereas the 2 in the exponent of the $r^2$ is simply the Hausdorff dimension of $\R^2$.

In order for the term $\int_0^\infty \int_{\R^2} \beta_E(x,r)^2 ~dx ~r^{-2} dr$ to be finite, one must have that $\beta_E(x,r)$ is small for ``most'' balls in the sense that the singular integral over $x$ and $r$ is finite.  Thus, one can view this as a quantitative version of Rademacher's theorem, which says rectifiable curves have tangents almost everywhere.  Note that there are two directions for the first statement of the theorem.  The ``only if'' direction is the necessary direction and the ``if'' direction is the sufficient.

Since Jones's original paper, the traveling salesman theorem has been generalized to higher dimensional Euclidean spaces $\R^d$ \cite{Ok-TSP} and even Hilbert space $\ell_2$ \cite{Schul-TSP}.  See \cite{bishop-peres, pajot} for more context and additional applications of these theorems. Recently, researchers have even studied variants of the analyst's traveling salesman problem for H\"older curves \cite{BNV}, although we still do not have a complete picture.

As for addressing the analyst's traveling salesman problem in non-Euclidean spaces, much of the effort has been in the setting of Carnot groups \cite{CLZ}, and in particular the Heisenberg group \cite{FFP, juillet, li-schul-1, li-schul-2}, although there has also been work done in general metric spaces \cite{hahlomaa, schul, david-schul-2} and certain fractal spaces \cite{david-schul}.

We will give an overview of Carnot groups in the following section.  For now, we simply say that they are a special class of nilpotent Lie groups that are topologically just some $\R^n$ and the reader can view as (possibly) noncommutative versions of vector spaces.  Importantly, they contain a distinguished set of lines called the {\it horizontal lines} that will be the analogue of affine lines in $\R^n$.  Due to the presence of these horizontal lines, one can define a na\"ive analogue of the $\beta$ numbers by simply infimizing over horizontal, rather than affine, lines:
\begin{align*}
  \beta_E(x,r) = \inf_L \sup_{z \in E \cap B(x,r)} \frac{d(z,L)}{r}.
\end{align*}
As mentioned above, $L$ are horizontal lines and $d$ will be any metric induced by a homogeneous norm (to be defined later).  The ball $B(x,r)$ is a ball as defined by the metric $d$.

Results in Carnot groups have been partial.  For the Heisenberg group $\H$ (the simplest nonabelian Carnot group) there has been an almost tight characterization.

\begin{theorem}[\cite{li-schul-1,li-schul-2}] \label{th:li-schul}
  Let $E \subset \H$ be a subset of the first Heisenberg group.  Define
  \begin{align*}
    \beta_p(E) := \diam(E) + \int_0^\infty \int_\H \beta_E(x,r)^p ~dx \frac{dr}{r^4}.
  \end{align*}
  If $E$ lies on a finite length rectifiable curve, then $\beta_4(E) < \infty$.  On the other hand, if there is any $p < 4$ so that $\diam(E) + \beta_p(E) < \infty$, then $E$ lies on a finite length rectifiable curve.
\end{theorem}
As $\beta \in [0,1]$, having integrability of a lower power of $\beta$ is a stronger condition.  It is unknown whether finiteness of $\diam(E) + \beta_4(E)$ is sufficient for $E$ to lie on a rectifiable curve.

For general Carnot groups, there is only a one-sided result.

\begin{theorem}[\cite{CLZ}]
  Let $E \subset G$ be a subset of a step $s$ Carnot group $G$ of Hausdorff dimension $Q$.  If $E$ lies on a finite length rectifiable curve, then
  \begin{align*}
    \diam(E) + \int_0^\infty \int_G \beta_E(x,r)^{2s^2} ~dx \frac{dr}{r^Q} < \infty.
  \end{align*}
  If $s = 2$, we can take the exponent of $\beta$ to be 4.
\end{theorem}

The challenge for the sufficient direction has been that the direction of movement for Carnot groups are far more heterogeneous than Euclidean spaces, which cannot be detected by the $\beta$'s.  In the next section, we will exhibit a Carnot group for which the necessary and sufficient directions of the traveling salesman theorem with the original $\beta$s must have a gap.

\begin{proposition} \label{p:gap}
  There exists a Carnot group $G$ with Hausdorff dimension 6 so that for every $p_1 < 4$ and $p_2 > 2$ there are compact curves $E_1, E_2 \subset G$ such that $E_1$ is rectifiable and $E_2$ is not, but
  \begin{align*}
    \int_G \int_0^\infty \beta_{E_1}(x,r)^{p_1} ~\frac{dr}{r^6} ~dx &= \infty, \\
    \int_G \int_0^\infty \beta_{E_2}(x,r)^{p_2} ~\frac{dr}{r^6} ~dx &< \infty.
  \end{align*}
\end{proposition}

Thus, we see that we need a new kind of $\beta$-number if we are to have any hope of a sharp traveling salesman theorem that works for every Carnot group.  Note that this proposition does not preclude the possibility that the usual $\beta$-based traveling salesman theorem might be tight for certain Carnot groups (like $\H$).

Given a step $s$ Carnot group, we can form a projective system by quotienting out the center.  We let $G_s = G$ and get projections $G_i \to G_{i-1}$ until $G_0 = \{0\}$.  Let $\pi_i : G \to G_i$ be the composition of the projections from $G_s$ to $G_i$.

We now define a new type of $\beta$-number that we call stratified $\beta$-numbers.  For a subset $E \subset G$ and a ball $B(x,r)$, define:
\begin{align}
  \hat{\beta}_E(x,r)^{2s} &:= \inf_L \sum_{i=1}^s \sup_{z \in E \cap B(x,r)} \left(\frac{d(\pi_i(z),\pi_i(L))}{r}\right)^{2i} \label{e:orig-def}
\end{align}
where the infimum is over horizontal lines.  We have chosen the exponent so that $\hat{\beta}_E \geq \beta_E$.  One can replace the $\ell_1$-sum by an $\ell_p$ combination of the summands to get an equivalent notion of $\hat{\beta}$.  Note that we have abused notation here as $d$ is originally a metric on $G$, not the projected $\pi_i(G)$.  For now the reader can interpret this as $d$ being different homogeneous norms for every $\pi_i(G)$ (which will still give a valid $\hat{\beta}$) but we will clarify what we mean in the following section.

The following is the main result of this paper.
\begin{theorem} \label{th:main}
  Let $G$ be a step $s$ Carnot group with Hausdorff dimension $Q$.  For a set $E \subset G$, we define the quantity
  \begin{align*}
    \hat{\beta}(E) := \int_0^\infty \int_G \hat{\beta}_E(x,r)^{2s} ~dx \frac{dr}{r^Q}.
  \end{align*}
  Then $E$ lies on a finite length rectifiable curve if and only if $\diam(E) + \hat{\beta}(E)$ is finite.
  More precisely, there exists a constant $C > 1$ depending only on $d$ and the Euclidean structure $|\cdot|$ on $G$ so that:
  \begin{enumerate}[label=(\arabic*)]
    \item If $\Gamma$ is any curve containing $E$ then $\hat{\beta}(E) + \diam(E) \leq C\Hd^1(\Gamma)$,
    \item If $\hat{\beta}(E) + \diam(E) < \infty$, then there exists a curve $\Gamma \supset E$ for which $\Hd^1(\Gamma) \leq C (\hat{\beta}(E) + \diam(E))$.
  \end{enumerate}
\end{theorem}

We conclude with some equivalent formulations of $\hat{\beta}$ that might be of use in the future.  The original $\beta$ was described as the rescaled radius of the thinnest tube containing $E \cap B(x,r)$.  Thus, it can be rewritten as
\begin{align*}
  \beta_E(x,r) = \inf_L \inf \{ \epsilon > 0 : E \cap B_G(x,r) \subset L \cdot \delta_r(B_G(\epsilon)) \}.
\end{align*}
We have a similar description of $\hat{\beta}$, as well as a characterization using the Riemannian metric.

\begin{proposition} \label{p:euc-char}
  Let $E \subset G$ be a subset, $B_G(x,r)$ be a ball in $G$, and $\rho$ be a left-invariant Riemannian metric on $G$.  The following quantities are equivalent up to multiplicative constants depending only on $d$, $\rho$, and the Euclidean structure $|\cdot|$:
  \begin{enumerate}[label=(\arabic*)]
    \item $\hat{\beta}_E(x,r)$,
    \item $\inf_L \inf \{ \epsilon > 0 : E \cap B_G(x,r) \subset L \cdot \delta_r(B_{\R^n}(\epsilon^s)) \}$,
    \item $\inf_L \sup \{ \rho(z,L)^{1/s} : z \in \delta_{1/r}(x^{-1}E) \cap B_G(0,1) \}$.
  \end{enumerate}
  Here, the $L$ in the infimums of (2) and (3) vary over horizontal lines.
\end{proposition}
The normalization in (3) cannot be removed as $\rho$ does not behave well with regards to scaling.

One sees that $\hat{\beta}$ gives asymptotically tighter bounds than $\beta$ the lower the layer we are considering.  Proposition \ref{p:euc-char} will be proven in Section \ref{s:proof}.  The fact that the $B_{\R^n}(\epsilon^s)$ is a Euclidean ball is not terribly important.  The important thing is the way it scales with $\epsilon$.  It could be the ball of any norm on $\R^n$ of radius $\epsilon^s$.


Proofs of the traveling salesman theorem tend to follow a standard path.  The most important step is to establish the following Alexandrov-like curvature condition: if $p_1,p_2,p_3$ are in $B(x,r)$ and $d(p_i,p_j) \geq \alpha r$ for $i \neq j$, then there exist constants $c,C > 0$ depending on $\alpha$ so that
\begin{align}
  c \beta(x,r)^\gamma r \leq d(p_1,p_2) + d(p_2,p_3) - d(p_1,p_3) \leq C \beta(x,r)^\gamma r \label{e:alex}
\end{align}
where $\beta^\gamma$ is the appropriate $\beta$ for the situation.  Once the left hand inequality is established, one can apply a clever telescoping argument to derive the necessary direction (although a little more work is needed, see Section \ref{s:flat}).  If the right hand inequality is established, then one applies a farthest insertion algorithm (given in Jones's original paper) to derive the sufficient direction.  Thus, most of this paper will be devoted to proving results in the vein of \eqref{e:alex} for $\hat{\beta}_E^{2s}$ (see Proposition \ref{p:curvature} and Corollaries \ref{c:curvature}, \ref{c:our-cor} for the relevant results).  We will not go too in-depth for the subsequent steps, instead referring the reader to other papers where these steps are covered in detail.

\subsection{Acknowledgments}
The research presented here was supported by NSF grant DMS-1812879.  We would also like to thank the anonymous referees for their detailed review that caught many errors and helped improve the presentation of the paper, including substantial simplifications of many proofs.

\section{Preliminaries}

A Lie algebra $\g$ is stratified if it is nilpotent and can be decomposed into a direct sum of subspaces $\g = V_1 \oplus ... \oplus V_s$ for which $[V_1, V_j] = V_{j+1}$ for $j \geq 1$.  It is understood that $V_k = 0$ for $k > s$.  The layer $V_1$ is called the horizontal layer and the largest $s$ for which $V_s \neq 0$ is called the (nilpotency) step of $\g$.  We let $n_i = \dim(V_i)$ and $n = n_1 + ... + n_s$.  We can thus identify $\g$ with $\R^n$ where the direct sum $\g = \bigoplus V_i$ is orthogonal.  Note that our choice of basis for $\R^n \cong \R^{n_1 + ... + n_s}$ has been chosen to be consistent with the grading of $\g$.

A Carnot group $G$ is a simply connected Lie group whose Lie algebra is stratified.  The exponential map is a diffeomorphism between simply connected Lie groups and their Lie algebras, and so we can use $\exp$ to identify elements of $G$ using vectors of $\g \cong \R^n$.  This allows us to push the coordinates of $\R^n$ to $G$ (the so-called exponential coordinates).  Thus, we will use $(g_1,...,g_s) \in \R^n$ to write coordinates for elements $g \in G$ where $g_i$ are the components of $\exp^{-1}(g)$ in $V_i$.  The identity element of $G$ is still 0.

The subgroups $\exp(V_i \oplus V_{i+1} \oplus ... \oplus V_s)$ are all normal.  For $i \in \{1,...,s\}$, we then define the projection maps $\pi_i : G \to G/\exp(V_{i+1} \oplus ... \oplus V_s)$.  In exponential coordinates, this amounts to the usual projection of $\R^n$ onto $\R^{n_1 + ... + n_i}$.  The image of projection $\pi_i$ is also a Carnot group and its Lie algebra is simply $\g / \bigoplus_{j=i+1}^s V_j$.

We use $\|\cdot\|$ to denote the standard Euclidean norm on $G$ when viewed as $\R^n$ through exponential coordinates.  This will depend upon an arbitrary choice of basis for the $V_i$s.  We can then also define centered Euclidean balls in $G$ as such:
\begin{align*}
  B_{\R^n}(\epsilon) = \{g \in G : \|g\| < \epsilon\}.
\end{align*}

The group multiplication in $G$ can be expressed at the level of the Lie algebra via the Baker-Campbell-Hausdorff (BCH) formula:
\begin{align}
  \log(\exp(U)\exp(V)) = \sum_{k > 0} \frac{(-1)^{k-1}}{k} \underset{\substack{r_i+s_i > 0,\\r_i,s_i \geq 0,\\1 \leq i \leq k}}{\sum} a(r_1,s_1,...,r_k,s_k) [U^{r_1} V^{s_1} \cdots U^{r_k} V^{s_k}] \label{e:BCH-form}
\end{align}
where the bracket term denotes iterated Lie brackets
\begin{align*}
  [ X^{r_1} Y^{s_1} \dotsm X^{r_n} Y^{s_n} ] = [\underbrace{X,[X,\dotsm[X}_{r_1} ,[ \underbrace{Y,[Y,\dotsm[Y}_{s_1} ,\,\dotsm\, [ \underbrace{X,[X,\dotsm[X}_{r_n} ,[ \underbrace{Y,[Y,\dotsm Y}_{s_n}].
\end{align*}
One can also translate the BCH formula into the exponential coordinates as
\begin{align*}
  (u_1,...,u_s) \cdot (v_1,...,v_s) = (u_1 + v_1, u_2 + v_2 + P_2,..., u_s + v_s + P_s)
\end{align*}
where each $P_i$ are polynomials of $u_1,...,u_{i-1},v_1,...,v_{i-1}$.  Recall that each $u_i,v_i$ are elements of $\R^{n_i}$.  We call these $P_i$s the BCH polynomials.

For example, the Heisenberg group $\H$ has a Lie algebra that is generated by $X,Y,Z$ with the bracket relations $[X,Y] = Z$, $[X,Z] = [Y,Z] = 0$.  We identify each element of $\H$ by $(a,b,c) = \exp(aX + bY + cZ)$.  Thus, the BCH formula says that
\begin{align*}
  (a,b,c) \cdot (a',b',c') &= \exp\left(aX + bY + cZ + a'X + b'Y + c'Z + \frac{1}{2} [aX + bY + cZ, a'X + b'Y + c'Z] \right) \\
  &= \exp((a + a')X + (b + b')Y + (c + c' + (ab' - a'b)/2)Z) \\
  &= (a+a',b+b',c+c' + (ab' - a'b)/2).
\end{align*}
In the first line, we used the fact that all the $k > 2$ terms of \eqref{e:BCH-form} are zero due to the nilpotency of the Lie algebra.  Thus, $P_2((a,b),(a',b')) = \frac{1}{2}(ab' - a'b)$ in this case.

An important fact about BCH polynomials is that $P_k(u,v)$ are composed of monomials of iterated Lie brackets of terms of $u = (u_1,...,u_s)$ and $v = (v_1,...,v_s)$ whose layers add up to $k$.  For example, $P_3$ is a linear combination of the monomials
$$[u_1,[u_1,v_1]], [v_1,[u_1,v_1]], [u_1,v_2], [v_1,u_2].$$

Importantly, for every $\lambda > 0$, there exists an automorphism $\delta_\lambda : G \to G$, which in exponential coordinates is given by
\begin{align*}
  \delta_\lambda(g_1,...,g_s) = (\lambda^1 g_1, \lambda^2 g_2,...,\lambda^s g_s).
\end{align*}

A homogeneous norm is a function $N : G \to [0,\infty)$ for which
\begin{enumerate}[label=(\arabic*)]
  \item $N(g^{-1}) = N(g)$,
  \item $N(\delta_\lambda(g)) = \lambda N(g)$, for all $g \in G$, $\lambda > 0$,
  \item $N(g) = 0 \Leftrightarrow g = 0$,
  \item $N(gh) \leq N(g) + N(h)$.
\end{enumerate}
Property (4) says that the norm is subadditive and is not always standard in the definition of homogeneous norm.  We include it because then one can define, for any homogeneous norm, a corresponding homogeneous metric $d(g,h) = N(g^{-1}h)$, which will be a left-invariant metric.  If subadditivity were not present, then one cannot guarantee the triangle inequality.  When we want to stress the subadditivity property of the homogeneous norm, we will use the term {\it subadditive homogeneous norm}.  Property (2) of homogeneous norms tells us that the $\delta_\lambda$ automorphisms scale the metric:
$$d(\delta_\lambda(g), \delta_\lambda(h)) = \lambda d(g,h).$$
We remark that it is well known that any two homogeneous norms (subadditive or not) on a Carnot group are biLipschitz equivalent.

In this paper, we will work primarily with homogeneous metrics, and specifically, with the metrics corresponding to two (classes of) subadditive homogeneous norms that exist for every Carnot group.  The first norm was introduced by Hebisch and Sikora in \cite{hebisch-sikora}.  For a fixed Euclidean ball $B_{\R^n}(\eta) \subset \R^n \cong G$ centered at the origin, one can define the Minkowski gauge
\begin{align*}
  N_{HS}(g) = \inf \{ \lambda > 0 : \delta_{1/\lambda}(g) \in B_{\R^n}(\eta) \}.
\end{align*}
As the balls $B_{\R^n}(\eta)$ are not homothetic under the dilation $\delta_\lambda$, we get truly different norms for different $\eta$.  The main result in \cite{hebisch-sikora} is that for sufficiently small $\eta$, $N_{HS}$ is a subadditive homogeneous norm that is smooth on $G \setminus \{0\}$.  We use $d_{HS}$ for the corresponding homogeneous metric.  The balls of $N_{HS}$ centered at the origin are then dilations of $B_{\R^n}(\eta)$ and so are axis-aligned ellipsoids.

Given a Carnot group $G$ and one of its projections $\pi_i(G)$, if we take $\eta > 0$ small enough, then $B_{\R^n}(\eta)$ and $B_{\R^{n_1 + ... + n_i}}(\eta)$ both can form Hebisch-Sikora norms.  In this way, we can assume that a $N_{HS}$, defined for $G$, is also a homogeneous norm for all of its projections $\pi_i(G)$.  This is how we make sense of the $d$ in $\hat{\beta}$.

Having fixed $\eta$, we also take $B_{\R^{n_1+...+n_i}}(\eta)$ as the unit ball of a new Euclidean norm $|\cdot|$ on each of the quotients $\pi_i(G) \cong \R^{n_1 + ... + n_i}$.  This is equivalent to defining $|\cdot| = \frac{1}{\eta}\|\cdot\|$ where $\|\cdot\|$ is the standard Euclidean metric.  Now $|g \pm h|$ makes sense for $g,h \in G$.  As the unit ball of $N_{HS}$ is now also the unit ball of $|\cdot|$, we then have that
\begin{align*}
  d_{HS}(\pi_1(g), \pi_1(h)) = |\pi_1(g) - \pi_1(h)|, \qquad \forall g,h \in G.
\end{align*}
From now on, all balls of the form $B_{\R^n}$ will be taken with respect to $|\cdot|$.

We remark that the maps $\pi_i$ are 1-Lipschitz as we have the relationship $N_{HS}(\pi_i(g)) \leq N_{HS}(g)$.  This gives us the following relation between $d_{HS}$ in $G$ and in its projections under $\pi_i$:
\begin{lemma}
  $$d_{HS}(\pi_i(g), \pi_i(h)) = \min_{g' \in \pi_i^{-1}(\pi_i(g)), h' \in \pi_i^{-1}(\pi_i(h))} d_{HS}(g',h'), \qquad \forall g,h \in G$$
\end{lemma}

\begin{proof}
  In one direction, we have $d_{HS}(g', h') \geq d_{HS}(\pi_i(g'),\pi_i(h')) = d_{HS}(\pi_i(g), \pi_i(h))$.

  In the other direction, let $u = \pi_i(g)^{-1}\pi_i(h) \in G_i \cong \R^{n_1 + ... + n_i}$.  We define the element $(u,0) \in G$ where $0 \in \R^{n_{i+1} + ... + n_s}$.  Given any $g' \in \pi_i^{-1}(g)$, we let $h' = g' \cdot (u,0)$.  Then
  $$d_{HS}(g',h') = N_{HS}(u,0) = N_{HS}(u) = d_{HS}(\pi_i(g), \pi_i(h)).$$
  Here, the first $N_{HS}$ is in $G$ and the second is in $G_i$.
\end{proof}

The second norm is analogous to the $\ell_\infty$ norm and is defined as
\begin{align}
  N_\infty(x_1,...,x_s) = \max_{1 \leq i \leq s} \lambda_i |x_i|^{1/i}, \label{e:infty-norm}
\end{align}
where $\lambda_i > 0$ are constants with $\lambda_1 = 1$.  Lemma 2.5 of \cite{breuillard} says that one can choose the $\lambda_i$s so that $N_\infty$ is a subadditive homogeneous norm.  The proof chooses $\lambda_i$ inductively based only on the BCH polynomial $P_i$ and $\{\lambda_j\}_{j=1}^{i-1}$.  Thus, we may choose $\{\lambda_j\}_{j=2}^s$ so that each set of coefficients $\{\lambda_j\}_{j=1}^i$ defines a subaddivitve homogeneous norms on $\pi_i(G)$.  We use $d_\infty$ for the corresponding homogeneous metric.

Both norms are known to be {\it quasi-convex}, that is, there is a constant $C_d > 0$ depending only on the metric so that for any two points $x,y \in G$, there is a rectifable curve $\gamma$ going from $x$ to $y$ with
\begin{align}
  \text{Length}(\gamma) \leq C_d d(x,y). \label{eq:qc}
\end{align}

Also, we can use left-translation to push-forward a given inner product on $\g \cong T_0G$ to get a left-invariant inner product on $TG$.  This allows us to define a Riemannian metric $\rho$.  It is known that given a $M > 0$, there exists a constant $C_0 > 0$ depending on $M$, $|\cdot|$, and $\rho$ so that origin centered Riemannian balls of radius at most $M$ are equivalent to origin centered Euclidean balls:
\begin{align}
  B_{\R^n}(\epsilon/C_0) \subseteq  B_\rho(\epsilon) \subseteq B_{\R^n}(C_0\epsilon), \qquad \forall \epsilon \in (0,M]. \label{e:riem-compare}
\end{align}

Elements of the exponential image of the horizontal layer, $\exp(V_1)$, are called horizontal elements of $G$.  In exponential coordinates, horizontal elements look like $(g_1,0,...,0)$.  Let $S^{n_1-1}$ be the unit sphere of $V_1 \subset \g$ with respect to the Euclidean metric $|\cdot|$.  Given any $v \in S^{n_1-1}$ and $g \in G$, we get an isometric embedding of $\R$ into $G$ by $t \mapsto g\exp(tv)$.  We call the image of such mappings {\it horizontal lines}.

We also define the map
\begin{align*}
  \tilde{\pi} : G &\to G \\
  (g_1,...,g_s) &\mapsto (g_1,0,...,0).
\end{align*}
Thus, $\tilde{\pi}$ maps $g$ to the horizontal element of $G$ that lies ``under'' $g$.  Note that this is not a homomorphism, nor it is Lipschitz.  However, we do have
\begin{align}
  N(\tilde{\pi}(g)) \leq N(g), \qquad \forall g \in G  \label{e:NH-shrink}
\end{align}
when $N \in \{N_{HS}, N_\infty\}$.  We can now define
\begin{align*}
  NH(g) = d(g, \tilde{\pi}(g)),
\end{align*}
which measures how non-horizontal an element of $g$ is.  Here, $d$ can be any homogeneous metric.

Finally, it is not hard to see from the BCH polynomials that the Jacobians of all left translations have determinant 1.  This tells us that the Lebesgue measure on the underlying manifold $\R^n$ of the Carnot group $G$ is also a Haar measure.

\section{The inadequacy of $\beta$}
In this section, we show that the non-stratified $\beta$s cannot provide sharp two-sided traveling salesman inequalities for certain Carnot groups.  This necessitates our use of a more complex $\beta$ in the sequel.

Consider the Carnot group $G = \R^2 \times \H$.  This can be viewed as the group $(\R^5,\cdot)$ where
\begin{align*}
  (x,y,u,v,w) \cdot (x',y',u',v',w') = \left(x+x', y+y', u+u', v+v', w+w' + \frac{1}{2} (uv' - u'v) \right).
\end{align*}
We can write elements of $G$ as $(z,g)$ where $z \in \R^2$ and $g \in \H_1$.  Then we endow $G$ with a product metric $d((z,g), (z',g')) = \max\{|z - z'|,d_\infty(g,g')\}$ where $d_\infty(g,g') = N_\infty(g^{-1}g')$ comes from the norm of \eqref{e:infty-norm}.  The Hausdorff dimension of $G$ is 6.

Let $G_1 = \H$ and $G_2 = \R^2$.  We can also view these as subgroups $H_1 = \{0\} \times \H$ and $H_2 = \R^2 \times \{0\}$ of $G$.  It is immediate from the construction of the metric on $G$ that $G_i$ is isometrically isomorphic to $H_i$.  Let $\iota_1 : \H \to H_1$ and $\iota: \R^2 \to H_2$ be the natural isomorphisms.  Since the $\beta$-based traveling salesman theorems for $\H$ and $\R^2$ have different exponents, we will use these two subgroups to exhibit a gap between the necessary and sufficient directions a $\beta$-based traveling salesman theorem for $G$.

Proposition \ref{p:gap} will follow from the following proposition.

\begin{proposition} \label{p:gap-1}
  For every $p_1 < 4$, there exists a finite length rectifiable curve $E_1 \subset H_1$ such that
  \begin{align*}
    \int_G \int_0^\infty \beta_{E_1}(x,r)^{p_1} ~\frac{dr}{r^6} ~dx = \infty.
  \end{align*}
  On the other hand, for every $p_2 > 2$, there exists a non-rectifiable curve $E_2 \subset H_2$ such that
  \begin{align*}
    \int_G \int_0^\infty \beta_{E_2}(x,r)^{p_2} ~\frac{dr}{r^6} ~dx < \infty.
  \end{align*}
\end{proposition}

Let $\tilde{\pi}_i : G \to H_i$ be the natural projection maps.  From the definition of the product metric $d$, we see that the $\tilde{\pi}_i$ are all 1-Lipschitz.  We first show that if a set $E$ lies in one of these subgroups, the best approximating horizontal line can also be taken to lie in the subgroup.

\begin{lemma} \label{l:equal-beta}
  Let $i \in \{1,2\}$.  If $E$ is a subset of $G_i$, then
  \begin{align*}
    \inf_{L \subset G} \inf_{z \in \iota_i(E)} d(z,L) = \inf_{L \subset G_i} \inf_{z \in E} d(z,L).
  \end{align*}
  Here, the $L$ in the infimums are horizontal lines in their respective groups.
\end{lemma}

\begin{proof}
  As $G_i$ and $H_i$ are isometrically isomorphic via $\iota_i$, it suffices to prove that if $E \subseteq H_i$, then
  \begin{align*}
    \inf_{L \subset G} \inf_{z \in E} d(z,L) = \inf_{L \subset H_i} \inf_{z \in E} d(z,L)
  \end{align*}
  where the infimum on the right hand side is over horizontal lines lying in $H_i$.  The $\leq$ inequality is clear.  Thus, we only need to prove the $\geq$ inequality.

  Let $L: \R \to G$ be the parameterization of a horizontal line in $G$ that we will also identify with its image.  Now define the new map $L' = \tilde{\pi}_i \circ L$.  This also parameterizes a horizontal line that is now a subset of $H_i$.  It suffices to show
  \begin{align*}
    d(z,L) \geq d(z,L'), \qquad \forall z \in E.
  \end{align*}
  Let $z \in E$ and let $t \in \R$.  We have by 1-Lipschitzness of $\tilde{\pi}_i$ that
  \begin{align*}
    d(z,L(t)) \geq d(\tilde{\pi}_i(z),\tilde{\pi}_i(L(t))) = d(z, L'(t)),
  \end{align*}
  where we used the construction of $L'$ and the fact that $z \in H_i$ in the last equality.  This proves that $d(z,L) \geq d(z,L')$, which proves the lemma.
\end{proof}

We now prove the following lemma, which says the Carleson integral of $\beta$s for sets lying in the subgroups $H_i$ behave as the Carleson integrals of $\beta$s for those sets in the model groups $G_i$.  This will allow us to prove Proposition \ref{p:gap} using previous known examples.

\begin{lemma} \label{p:subbeta}
  There are constants $c_1, c_2 > 0$ so that for any $p$,
  \begin{align*}
    \int_G \int_0^\infty \beta_{\iota_1(E)}(x,r)^p ~\frac{dr}{r^6} ~dx &= c_1 \int_{G_1} \int_0^\infty \beta_E(x,r)^p ~\frac{dr}{r^4} ~dx, \quad \forall E \subset G_1, \\
    \int_G \int_0^\infty \beta_{\iota_2(E)}(x,r)^p ~\frac{dr}{r^6} ~dx &= c_2 \int_{G_2} \int_0^\infty \beta_E(x,r)^p ~\frac{dr}{r^2} ~dx, \quad \forall E \subset G_2.
  \end{align*}
\end{lemma}

\begin{proof}
  We will prove the $G_1$ case of the lemma as the $i = 2$ case only requires trivial modifications of the forthcoming proof.

  Let $(z,g) \in G$ here $z \in \R^2$ and $g \in \H$.  We first claim that $B((z,g),r)$ intersects $H_1$ if and only if $|z| \leq r$ in which case
  \begin{align*}
    B((z,g),r) \cap H_1 = \iota_1(B(g,r)).
  \end{align*}
  Note that $H_1 = \{(0,h) : h \in \H\}$.  Thus, the first part of the claim follows as soon as we show that $d((z,g),H_1) = |z|$.  The $\leq$ direction follows as $(0,g) \in H_1$ and $d((z,g),(0,g)) = |z|$.  Now let $(0,h) \in H_1$.  Then
  \begin{align*}
    d((z,g),(0,h)) = \max\{|z|, d_\infty(g,h)\}  \geq |z|,
  \end{align*}
  which gives the $\geq$ direction as desired.

  As for the second part of the claim, we see that when $|z| \leq r$, then $(0,h) \in B((z,g),r) \cap H_1$ precisely when $d(g,h) \leq r$, that is $h \in B(g,r)$.  As $\{0\} \times B(g,r) = \iota_1(B(g,r))$, we get the second part of the claim.

  Now let $E \subset G_1$ and let $g \in \H$.  By the previous claim we have that $\beta_{\iota_1(E)}((z,g),r) = 0$ unless $|z| \leq r$ in which case
  \begin{align*}
    \beta_{\iota_1(E)}((z,g),r) &= \inf_{L \subset G} \inf_{y \in B((z,g),r) \cap \iota_1(E)} \frac{d(y,L)}{r} = \inf_{L \subset G} \inf_{y \in \iota_1(B(g,r) \cap E)} \frac{d(y,L)}{r} \\
    &= \inf_{L \subset G_1} \inf_{y \in B(g,r) \cap E} \frac{d(y,L)}{r} = \beta_E(g,r)
  \end{align*}
  where we used Lemma \ref{l:equal-beta} in the penultimate equality.

  We now integrate
  \begin{align*}
    \int_0^\infty \int_G \beta_{\iota_1(E)}(x,r) ~dx \frac{dr}{r^6} 
      &= \int_0^\infty \int_{\H} \int_{\R^2} \beta_{\iota_1(E)}((z,g),r) ~dz ~dg \frac{dr}{r^6} \\
      &= \int_0^\infty \int_{\H} \int_{\R^2} {\bf 1}_{\{|z| \leq r\}} \beta_E(g,r) ~dz ~dg \frac{dr}{r^6} \\
      &= c_1 \int_0^\infty \int_{\H} \beta_E(g,r) ~dg \frac{dr}{r^4}.
  \end{align*}
  This proves the lemma.
\end{proof}

We now prove Proposition \ref{p:gap-1}.

\begin{proof}[Proof of Proposition \ref{p:gap-1}]
  As mentioned at the beginning of this section, this will amount to recalling the corresponding $E_i \subset G_i$ so that
  \begin{align*}
    \int_{\H} \int_0^\infty \beta_{E_1}(x,r)^{p_1} ~\frac{dr}{r^4} ~dx &= \infty, \\
    \int_{\R^2} \int_0^\infty \beta_{E_2}(x,r)^{p_2} ~\frac{dr}{r^2} ~dx &< \infty
  \end{align*}
  for the chosen $p_1, p_2$.

  The $E_1$ will come from the counter example from \cite{juillet}, although three remarks need to be made.

  First, in \cite{juillet}, Juillet uses a special metric on $\H$, the Carnot-Carath\'eodory distance (that we will not introduce), whereas we are using the $N_\infty$ group norm.  However, as the Carnot-Carath\'eodory distance is induced by a homogeneous norm, the two metrics are biLipschitz equivalent.  As $\beta$-numbers change only by a nonzero multiplicative factor under biLipschitz change of metrics, the finiteness of the Carleson integrals are the same regardless of which metric we use.

  Second, Juillet sums $\beta$s over balls whose centers come from a system of nets in $\H$, but it is well known that summing $\beta$ over nets and Carleson integrating $\beta$ give quantities that differ only by a multiple depending on the choice of nets.  In particular, infiniteness of the quantity is preserved.

  Finally, the result of Juillet states our result only for when $p_1$ is chosen to be 2 (although the statement also holds for any $p_1 < 2$).  But a simple modification of the construction by choosing $\theta_n = \frac{C}{n^{2/p_1}}$ on p. 1046 of \cite{juillet} allows us to extend the example to any $p_1 < 4$.  The rest of the proof will remain unchanged except at the following spots.  For Lemma 3.2 on p. 1047, $L$, the length of the curve, will now be bounded by a constant that depends on $p_1$.  In the proof we have the modified bound:
  \begin{align*}
    \sum_{n=1}^N \theta_n^2 \leq \sum_{n=1}^N \frac{C^2}{n^{4/p_1}} < \infty
  \end{align*}
  as $p_1 > 4$.

  At the summation on p. 1052, we change $\beta^2$ to $\beta^{p_1}$ and get the lower bound for the sum of $\beta$s:
  \begin{align*}
    \sum_{k \in \N} 2^{-k} \sum_{x \in \Delta_k} \beta_\H^{p_1}(x, A \cdot 2^{-k}) (\omega([0,1])) &\geq \sum_{k \in \N} 2^{-k} 2^k \left( \frac{D_{\theta_{\lceil k/2 \rceil + 1}}}{A} \right)^{p_1} \\
                                                                                                   &\geq C^{p_1/2} \sum_{k \in \N} \frac{1}{\lfloor k/2 \rfloor + 1} = \infty.
  \end{align*}

  For $E_2$, one can use the example of Proposition 4.1 of \cite{BM}.  Again, the quantity calculated there was a sum of $\beta$ numbers over balls centered at a system of nets, but this is equivalent to our Carleson integral.
\end{proof}

\section{Lemmas}

In this section, we prove a few results that will be useful in the following sections.




The following lemma will allow us to bound BCH polynomials.

\begin{lemma} \label{l:BCH-bound}
  There exists some constant $C > 0$ depending only on $G$ and the Euclidean structure $|\cdot|$ so that if $|y_i| \leq \eta$ and $|x_i| \leq 1$ for all $i \in \{1,...,k-1\}$ and any $\eta \in (0,1)$, then
  \begin{align*}
    |P_k(x_1,...,x_{k-1},y_1,...,y_{k-1})| \leq C\eta.
  \end{align*}
\end{lemma}

\begin{proof}
  We first define the quantity $c = max\{1,\sup_{x,y \in B_{\R^n}(1)} |[x,y]|\}$, which is finite by boundedness of $B_{\R^n}(1)$.  It depends only on $G$ and the Euclidean metric.  We can conclude then that
  \begin{align}
    |[u,v]| \leq c|u||v|, \qquad \forall u,v \in G. \label{e:cross-bound}
  \end{align}

  $P_k$ is a polynomial of nested Lie bracket monomials.  As the number of summands in the polynomial is bounded by a constant dependent only on the algebraic structure of $G$, it suffices to prove each monomial is bounded by a constant of $\eta$.  Each Lie bracket monomial is of the form $[g_1,[g_2,...,[g_{j-1},g_j]...]]$ where $j \leq k$ and $g_i \in \{x_1,...,x_{k-1},y_1,...,y_{k-1}\}$.  It must be that one of $\{g_{j-1},g_j\}$ (say $g_j$) is in $\{y_1,...,y_{k-1}\}$.  Indeed, as the monomial is expanded from iterated Lie brackets of the vectors $x,y$, the inner bracket must be of the form $[x,y]$ as otherwise the iterated bracket would be 0 from antisymmetry of the bracket.

  We get
  $$|[g_1,[g_2,...,[g_{j-1},g_j]...]]| \overset{\eqref{e:cross-bound}}{\leq} c^{j-1} \prod_{i=1}^j |g_i| \leq c^{j-1} \eta \prod_{i=1}^{j-1} |g_i| \leq c^s \eta.$$
\end{proof}

The next few lemmas are the main results of this section.  They will allow us to conclude the following fact: if $g \in h \cdot B_{\R^n}(\epsilon)$ and $u,v \in S^{n_1-1}$ are so that $|u-v| < \epsilon$, then $g \exp(u) \in h\exp(v) \cdot B_{\R^n}(C\epsilon)$ for some constant $C$ depending only on the metric $d$ and the Euclidean metric on $G$.  Had we only had $d(g,h) \leq \epsilon$, then we could only guarantee $d(g\exp(u), h\exp(v)) \leq C\epsilon^{1/s}$ (see Lemmas 3.6, 3.7, and 3.8 of \cite{li-coarse}).

This essentially comes from the fact that we are group multiplying an element $g^{-1}h$ with entries of order $\epsilon$ by displacements of order 1.  One can see this by considering a simple case in the Heisenberg group.  Let $h = (0,0,0)$, $g = (0,\epsilon,0)$, and $u = v = (1,0,0)$.  Then
\begin{align*}
  (g\exp(u))^{-1}h\exp(v) = ((0,\epsilon,0)(1,0,0))^{-1}(1,0,0) = (0,0,\epsilon)
\end{align*}
Thus, while $|g^{-1}h|$ and $N(g^{-1}h)$ are both of order $\epsilon$, we have that $|(g\exp(u))^{-1}h\exp(v)|$ is of order $\epsilon$ while $N((g\exp(u))^{-1}h\exp(v))$ is of order $\epsilon^{1/2}$.

This behavior will be why stratified $\beta$-numbers are needed.

We first begin with a smoothness argument on products of elements.

\begin{lemma} \label{l:smoothness}
  There exists a $C > 0$ depending only on $G$ and the Euclidean structure $|\cdot|$ so that if $\eta \in (0,1)$, $p \in B_{\R^n}(\eta)$, and $u,v \in B_{\R^n}(1)$ are such that $|u-v| \leq \eta$, then
  \begin{align*}
    v^{-1}pu \in B_{\R^n}(C\eta).
  \end{align*}
\end{lemma}

\begin{proof}
  Define the function $\phi(v,p,u) = v^{-1}pu$.  This is a smooth function on $G \times G \times G$.  In particular, it is $L$-Lipschitz with respect to the Euclidean metric on $B_{\R^n}(2) \times B_{\R^n}(2) \times B_{\R^n}(2)$ for some $L > 0$ (depending on $|\cdot|$) and so
  \begin{align*}
    |\phi(v,p,u)| \leq |\phi(v,p,u) - \phi(u,p,u)| + |\phi(u,p,u) - \phi(u,0,u)| \leq L|v-u| + L|p| \leq 2L\eta.
  \end{align*}
\end{proof}

We can now prove the main result of this section.
\begin{lemma} \label{l:close-euc}
  There exists a $C > 0$ depending only on $G$ and the Euclidean structure $|\cdot|$ so that if $g,h \in G$ and $u,v \in B_{\R^{n_1}}(1)$ are such that $g \in h \cdot \delta_\ell(B_{\R^n}(\eta))$ and $|u-v| \leq \eta$ for some $\eta \in (0,1)$ and $\ell > 0$, then
  \begin{align*}
    g\delta_t(u) \in h\delta_t(v) \cdot \delta_\ell(B_{\R^n}(C\eta)), \qquad \forall t \in [0,\ell].
  \end{align*}
\end{lemma}

\begin{proof}
  We may dilate the setting to reduce to the case that $\ell = 1$.  Thus, let $t \in [0,1]$ and $u' = \delta_t(u)$ and $v' = \delta_t(v)$.  We need to prove that
  \begin{align*}
    (v')^{-1}h^{-1} gu' \in B_{\R^n}(C\eta)
  \end{align*}
  for some $C$.

  As $u,v$ lie in the first layer $\R^{n_1}$, we get that $u' = tu$ and $v' = tv$ and so we still have $|u' - v'| = t|u-v| \leq \eta$.  Note that $h^{-1}g \in B_{\R^n}(\eta)$ and $u',v' \in B_{\R^{n_1}}(1)$.  Thus, the result follows from applying Lemma \ref{l:smoothness} with $v'$, $h^{-1}g$, and $u'$.
\end{proof}

We now relate the relation between two points of $G$ in terms of containment by Euclidean balls with a $\hat{\beta}$-style bound on their distance.  The following lemma applies to both $d_\infty$ and $d_{HS}$.

\begin{lemma} \label{l:beta-balls}
  Given $M > 0$, there exists $C > 0$ depending only on $M$, $d$, and the Euclidean structure $|\cdot|$ so that $p,q \in G$ are such that $p \in q \cdot \delta_\ell (B_{\R^n}(\eta))$ for some $\eta \in (0,M]$ and $\ell > 0$ if and only if
  $$\max_{1 \leq i \leq s} \left(\frac{d(\pi_i(p),\pi_i(q))}{\ell}\right)^i \leq C\eta.$$
\end{lemma}

\begin{proof}
  As the $\pi_i$ are homomorphisms, we may translate and dilate to reduce to the case $\ell = 1$ and $q = 0$.  We will prove the lemma for $d_\infty$ as that allows us to conclude the same statement for $d_{HS}$ by a biLipschitz change of metric (for a different $C$).

  We are now asked to show that $p \in B_{\R^n}(\eta)$ is equivalent to
  \begin{align*}
    \max_{1 \leq i \leq s} N_\infty(\pi_i(p))^i \leq C\eta.
  \end{align*}

  If $p \in B_{\R^n}(\eta)$, then $|p_i| \leq C_0\eta$ for all $i$ for some $C_0$ depending on the Euclidean metric.  As $\eta \leq M$, we have for any $i$ that
  \begin{align*}
    N_\infty(\pi_i(p))^i = N_\infty(p_1,...,p_i)^i = \max_{1 \leq j \leq i} \lambda_j^i |p_i|^{i/j} \leq \eta \max_{1 \leq j \leq i} \lambda_j^i C_0^{i/j} M^{i/j-1}.
  \end{align*}
  The result then holds for some $C$ depending on the $\lambda_i$, $M$, and $C_0$.

  Now assume $N_\infty(\pi_i(p))^i \leq \eta$ for all $i$.  Then for any $i$
  \begin{align*}
    \lambda_i^i |p_i| \leq N_\infty(\pi_i(p))^i \leq \eta
  \end{align*}
  and so $|p_i| \leq C_1 \eta$ for some $C_1 > 0$ depending on the $\lambda_i$.  The result then holds for some $C$ depending on the $\lambda_i$ and the Euclidean metric.
\end{proof}

The following corollary is an immediate consequence of the previous two lemmas.

\begin{corollary} \label{l:close-lines}
  There exists a constant $C > 0$ depending only on $d_{HS}$ and the Euclidean structure $|\cdot|$ so that if $g,h \in G$ and $u,v \in B_{\R^{n_1}}(1)$ are such that $g \in h \cdot \delta_\ell(B_{\R^n}(\eta))$ and $|u-v| \leq \eta$ for some $\eta \in (0,1)$ and $\ell > 0$, then
  \begin{align*}
    \sup_{t \in [0,\ell]} \left( \frac{d_{HS}(\pi_i(g \cdot (tu,0,...,0)),\pi_i(h \cdot (tv,0,...,0)))}{\ell} \right)^i \leq C \eta, \qquad \forall i \in \{1,...,s\}.
  \end{align*}
\end{corollary}

\section{Necessity}

We begin by proving the necessary direction of Theorem \ref{th:main}.  This will be achieved by proving the following theorem.

\begin{theorem} \label{th:necessity}
  Let $G$ be a step $s$ Carnot group with Hausdorff dimension $Q$ and $\Gamma \subset G$, a finite length rectifiable curve.  Then there is some constant depending only on $G$ and its Euclidean structure $|\cdot|$ so that
  \begin{align*}
    \diam(\Gamma) + \int_0^\infty \int_G \hat{\beta}_\Gamma(x,r)^{2s} ~dx \frac{dr}{r^Q} \leq C \ell(\Gamma).
  \end{align*}
\end{theorem}

Note that the bound on the $\diam(\Gamma)$ term is obvious.  In fact, we may assume without loss of generality $\diam(\Gamma) = 1$ as $\hat{\beta}_{\delta_\lambda(E)}(\delta_\lambda(x), \lambda r) = \hat{\beta}_E(x,r)$.  Indeed, we have that $z \in E \cap B(x,r)$ if and only if $\delta_\lambda(z) \in \delta_\lambda(E) \cap B(\delta_\lambda(x),\lambda r)$.  Now given any horizontal line $L$, we get
\begin{align*}
  \frac{d(\pi_i(\delta_\lambda(z)), \pi_i(\delta_\lambda(L)))}{\lambda r} = \frac{d(\delta_\lambda(\pi_i(z)), \delta_\lambda(\pi_i(L)))}{\lambda r} = \frac{d(\pi_i(z), \pi_i(L))}{r}.
\end{align*}

For $g,h \in G$ and $t \in [0,1]$, define
\begin{align*}
  L_{g,h}(t) := g \delta_t(\tilde{\pi}(g^{-1}h)).
\end{align*}
This is a horizontal line segment that starts from $g$ and goes in the horizontal direction towards $h$ (although it may not hit $h$).  We also write
$$\overline{g,h} = \{L_{g,h}(t)\}_{t \in [0,1]}.$$

We first prove the following simple bound.
\begin{lemma} \label{l:lines-bounds}
  Let $p_1,p_2,... \in G$ and $\ell > 0$ be so that $d_{HS}(p_i,p_j) \leq \ell$ for all $i, j$.  Then for any $i \in \{1,...,s\}$, we have that
  \begin{align*}
    \sup_{t \in [0,1]} d_{HS}(\pi_i(L_{p_j,p_k}(t)), \pi_i(\overline{p_m,p_n})) \leq 2\ell, \qquad \forall j,k,m,n.
  \end{align*}
\end{lemma}

\begin{proof}
  As $\pi_i$ is 1-Lipschitz, it suffices to prove the inequality without the $\pi_i$ present.  As $N_{HS}(\tilde{\pi}(g)) \leq N_{HS}(g)$, we get for any $i,j$ that
  \begin{align*}
    d_{HS}(p_j,L_{p_j,p_k}(t)) = d_{HS}(0, \delta_t(\tilde{\pi}(p_j^{-1}p_k))) = t N_{HS}(\tilde{\pi}(p_j^{-1}p_k)) \leq t\ell \leq \ell
  \end{align*}
  when $t \in [0,1]$.  As $p_m \in \overline{p_m,p_n}$, we get
  \begin{align*}
    d_{HS}(L_{p_j,p_k}(t), \overline{p_m,p_n}) \leq d_{HS}(p_j, L_{p_j,p_k}(t)) + d_{HS}(p_j,p_m) \leq 2\ell.
  \end{align*}
\end{proof}

\subsection{Stratified curvature}
As mentioned in the introduction, we will be interested in proving inequalities for the triangle inequality excess:
\begin{align*}
  \Delta_i(p_1,...,p_k) = \sum_{j=1}^{k-1} d_{HS}(\pi_i(p_j),\pi_i(p_{j+1})) - d_{HS}(\pi_i(p_1),\pi_i(p_k)).
\end{align*}
It is important to note that this quantity is not preserved in any good sense under a change of biLipschitz equivalent metrics.  Thus, most of the results in this subsection will only be valid for $d_{HS}$.  However, these results are used to prove the necessary direction of the main theorem which is biLipschitz invariant.

Given $\lambda \in (0,1)$, $\ell > 0$, and $i \in \{1,...,s\}$, we define
\begin{align*}
  A(\lambda,\ell,k) &= \{(g,h) \in G \times G : N_{HS}(\pi_k(g)), N_{HS}(\pi_k(h)), N_{HS}(\pi_k(gh)) \in (\lambda \ell, \ell)\}, \\
  \overline{A}(\lambda,\ell,k) &= \{(g, h) \in G \times G : N_{HS}(\pi_k(g)), N_{HS}(\pi_k(h)), N_{HS}(\pi_k(gh)) \in [\lambda \ell, \ell]\}.
\end{align*}
Note that $A(\lambda,\ell,k)$ and $\overline{A(\lambda,\ell,k)}$ are open and compact in $G \times G$ respectively.

We will first need the following lemma from \cite{CLZ} that bounds $\Delta_s(p_1,p_2,p_3)$ from below when the points $p_i$ are sufficiently well separated.

\begin{lemma} \label{l:CLZ}
  For every $\lambda \in (0,1)$, there exists a constant $C > 0$ depending only on $d_{HS}$ and $\lambda$ so that if $p_1,p_2,p_3 \in G$ are so that $\lambda \ell \leq d_{HS}(p_i,p_j) \leq \ell$ when $i \neq j$ for some $\ell > 0$, then
  \begin{align*}
    \frac{d_{HS}(\pi_1(p_2),\pi_1(\overline{p_1,p_3}))^2}{\ell} + \frac{NH(p_1^{-1}p_3)^{2s}}{\ell^{2s-1}} \leq C\Delta_s(p_1,p_2,p_3).
  \end{align*}
\end{lemma}

\begin{proof}
  This is essentially Lemma 3.4 of \cite{CLZ}, but we will explain the conversion, as \cite{CLZ} uses slightly different notation.  In \cite{CLZ}, $r$ is the step of $G$, the constant $m \in [0,1]$ can be chosen arbitrarily as stated right before Lemma 3.4, and $\alpha$ is a constant depending only on the group and metric structure of $G$ as defined in (3.1) of \cite{CLZ}.  The $\R^n$ is the first layer and so corresponds to our $\R^{n_1}$.  The $x_1,y_1$ corresponds to our projections $\pi_1(x)$ and $\pi_1(y)$ in $\R^{n_1}$ and the $\|\cdot\|$ denotes the Hebisch-Sikora norm in the Carnot group.  The $NH(g)$ in \cite{CLZ} is just the group term $\tilde{\pi}(g)^{-1}g$, and so $\|NH(g)\|$ is equivalent to our $NH(g)$.

  For $u \in \R^{n_1}$ the $\ell_u$ is the line in $\R^{n_1}$ from the origin to $u$.  Note that this is not the same as our $\pi_1(\overline{p,q})$ as this line segment starts from $\pi_1(p)$ and goes to $\pi_1(q)$.  However, we do have that
  $$\pi_1(\overline{p,q}) = \pi_1(p) + \ell_{\pi_1(p^{-1}q)}.$$
  One now simply sets $x = p_1^{-1}p_2$ and $y = p_2^{-1}p_3$ and uses the fact that $x_1 + y_1 = \pi_1(p_1^{-1}p_3)$ (the left hand side uses the projection notation of \cite{CLZ} but the right hand side does not) to get that
  \begin{align*}
    \ell_{x_1+y_1} - x_1 = (\pi_1(p_1) + \ell_{\pi_1(p_1^{-1}p_3)}) - (\pi_1(p_1) + \pi_1(p_1^{-1}p_2)) = \pi_1(\overline{p_1,p_3}) - \pi_1(p_2)
  \end{align*}
  where we are interpreting $\ell_{x_1+y_1}$ as a set.  Thus, we have that
  \begin{align*}
    d_{\R^n}(x_1, \ell_{x_1+y_1}) = d_{HS}(\pi_1(p_2), \pi_1(\overline{p_1,p_3}))
  \end{align*}
  where the left hand side is written in the notation of \cite{CLZ}.

  To get our lemma, we first bound
  \begin{align}
    NH(p_1^{-1}p_3) &\leq N_{HS}(p_1^{-1}p_3) + N_{HS}(\tilde{\pi}(p_1^{-1}p_3)) \overset{\eqref{e:NH-shrink}}{\leq} 2\ell, \label{e:A1-bound} \\
    d_{HS}(\pi_1(p_2), \pi_1(\overline{p_1,p_3})) &\leq d_{HS}(\pi_1(p_2), \pi_1(p_1)) \leq \ell. \label{e:A2-bound}
  \end{align}
  As $d_{HS}(p_i,p_j) \geq \lambda \ell$, we get that the $A$ term in \cite{CLZ} is bounded by
  \begin{align*}
    A \leq \frac{m^{2s}}{16} \max\{\alpha, 1\} (2\lambda^{1-2s} + (2\lambda)^{-1})\ell
  \end{align*}
  Thus, by choosing $m$ sufficiently small depending on $\alpha$ and $\lambda$, we get that $d_{HS}(p_i,p_j) \geq \lambda \ell \geq 4A$ for all $i \neq j$ for the hypothesis of the lemma.  The conclusion of Lemma 3.4 of \cite{CLZ} now gives our lemma (remembering again that $d(p_i,p_j) \geq \lambda \ell$).
\end{proof}

We will need to augment this bound as it is not sufficient for our goals.

First, we will prove a lemma that says if the projection under $\pi_k$ of well-separated elements of $G$ are no longer well separated, then the points must have a large $\Delta_s$.

\begin{lemma} \label{l:flat-points}
  For each $\lambda \in (0,1)$ and $k \in \{1,...,s-1\}$ there exists a constant $C > 0$ depending also on $d_{HS}$ so that if $(u,v) \in \overline{A}(\lambda,\ell,s) \backslash A(\lambda/2,\ell,k)$ then
  \begin{align*}
    N_{HS}(u) + N_{HS}(v) - N_{HS}(uv) \geq C\ell.
  \end{align*}
\end{lemma}

\begin{proof}
  Assume such a $C$ does not exist.  Then one can find a sequence $(u_n,v_n)$ for which
  $$N_{HS}(u_n) + N_{HS}(v_n) - N_{HS}(u_nv_n) \to 0.$$
  As $\overline{A}(\lambda,\ell,s) \backslash A(\lambda/2,\ell,k)$ is compact, we can pass to a limit pair $(u,v) \in \overline{A}(\lambda,\ell,s)$ for which
  \begin{align}
    N_{HS}(u) + N_{HS}(v) - N_{HS}(uv) = 0. \label{e:triangle-eq}
  \end{align}
  By Lemma \ref{l:CLZ}, we then have that $NH(uv) = 0$.  This means $uv$ is precisely the horizontal element $uv = (\pi_1(u) + \pi_1(v),0,...,0)$ and that $N_{HS}(uv) = |\pi_1(u) + \pi_1(v)|$.

  As the Hebisch-Sikora norm is 1-Lipschitz under $\pi_1$ and isometric only for horizontal elements, we have that $u$ and $v$ must be horizontal as otherwise \eqref{e:triangle-eq} will be positive.  However, this would mean that $N_{HS}(\pi_k(u)) = N_{HS}(u)$ and $N_{HS}(\pi_k(v)) = N_{HS}(v)$, which would mean $(u,v) \in A(\lambda/2,\ell,k)$, a contradiction.
\end{proof}

We now start controlling the triangle inequality excess from below.  In the following lemma, we seek to apply Lemma \ref{l:CLZ} to projections of points under $\pi_k$.  However, if the projected points are no longer well separated, one of the requirements of Lemma \ref{l:CLZ}, then we can apply Lemma \ref{l:flat-points} instead.

\begin{lemma} \label{l:curvature}
  For every $\lambda \in (0,1)$ and $k \in \{2,...,s\}$, there exists a constant $C > 0$ depending only on $d_{HS}$ and $\lambda$ so that if $p_1,p_2,p_3 \in G$ are so that $\lambda \ell \leq d_{HS}(p_i,p_j) \leq \ell$ when $i \neq j$ for some $\ell > 0$, then
  \begin{align*}
    \frac{d_{HS}(\pi_1(p_2),\pi_1(\overline{p_1,p_3}))^2}{\ell} &+ \frac{NH(\pi_k(p_1^{-1}p_3))^{2k}}{\ell^{2k-1}} \leq C(\Delta_k(p_1,p_2,p_3) + \Delta_s(p_1,p_2,p_3)).
  \end{align*}
\end{lemma}

\begin{proof}
  Let $u = p_1^{-1}p_2$ and $v = p_2^{-1}p_3$.  By assumption, $(u,v) \in \overline{A}(\lambda,\ell,s)$.  If $(u,v) \in A(\lambda/2,\ell,k)$, then we apply Lemma \ref{l:CLZ} to the points $\pi_k(p_j)$ in the Carnot group $\pi_k(G)$ to get an upper bound of $C\Delta_k(p_1,p_2,p_3)$ for the left hand side of the desired inequality.

  Thus, we may now assume that $(u,v) \in \overline{A}(\lambda,\ell,s) \backslash A(\lambda/2,\ell,k)$.  Then Lemma \ref{l:flat-points} tells us there is a constant $C_0 > 0$ so that
  \begin{align*}
    \Delta_s(p_1,p_2,p_3) \geq C_0\ell.
  \end{align*}
  If we show
  \begin{align}
    \frac{d_{HS}(\pi_1(p_2),\pi_1(\overline{p_1,p_3}))^2}{\ell} + \frac{NH(\pi_k(p_1^{-1}p_3))^{2k}}{\ell^{2k-1}} \leq 2^{2k+1}\ell \label{e:first-option}
  \end{align}
  then the lemma follows with $C = 2^{2s+1} C_0^{-1}$.  This follows from \eqref{e:A2-bound} and a bound similar to \eqref{e:A1-bound}, but using the fact that $\pi_k$ is 1-Lipschitz.
\end{proof}

We now combine the curvature bounds of Lemma \ref{l:curvature} across all layers to get the following new stratified curvature bound.  One can now see where the $\hat{\beta}$ starts appearing.

\begin{proposition} \label{p:curvature}
  For any $\lambda > 0$, there exists a constant $C > 0$ depending only on $d_{HS}$, $\lambda$, and Euclidean structure $|\cdot|$ so that the following holds.  Let $p_1,p_2,p_3,p_4 \in G$ be points so that $\lambda \ell \leq d_{HS}(p_i,p_j) \leq \ell$ if $i \neq j$ for some $\ell > 0$.  Then
  \begin{align*}
    \sum_{i=1}^s \sup_{t \in [0,1]} \frac{d_{HS}(\pi_i(L_{p_1,p_3}(t)),\pi_i(\overline{p_1,p_4}))^{2i}}{\ell^{2i-1}} +  \sup_{t \in [0,1]} \frac{d_{HS}(\pi_i(L_{p_3,p_4}(t))),\pi_i(\overline{p_1,p_4}))^{2i}}{\ell^{2i-1}} \leq C \Delta,
  \end{align*}
  where $\Delta = \sum_{i=2}^s \Delta_i(p_1,p_2,p_3,p_4)$.
\end{proposition}

It may seem strange that $p_2$ does not appear in the left hand side of the conclusion but does for the right.  The $p_2$ will come from applying Lemma \ref{l:curvature} to the points $p_1,p_2,p_3$ to get a bound on $NH(p_1^{-1}p_3)$.  This, combined with bounds on $d_{HS}(\pi_1(p_3), \pi_1(\overline{p_1,p_4}))$ that we get from applying Lemma \ref{l:curvature} to $p_1,p_3,p_4$ allows us to bound how far $p_3$ is to the horizontal line segment $\overline{p_1,p_4}$ so that we can get the second term bound on $L_{p_3,p_4}$ via Corollary \ref{l:close-lines}.  This second term bound is required in the proof of Corollary \ref{c:curvature} that follows this proposition.

Note that it is impossible to derive such a bound using Lemma \ref{l:curvature} directly on $p_1,p_3,p_4$ as it only bounds the behavior of the projection $\pi_1(p_3)$.  To achieve this, we would need a stronger version of Lemma \ref{l:CLZ}, but we do not seek to do so as it is not necessary.

\begin{proof}


  We may translate and dilate the setting so that $p_1 = 0$ and $\ell=1$.  Lemma \ref{l:lines-bounds} gives that the left hand side is then bounded by $s2^{2s+1}$.  Fix a $\mu > 0$ to be determined. If $\Delta \geq \mu$, then we have that the inequality holds for $C = s2^{2s+1}/\mu$.  Thus, we may suppose $\Delta < \mu$.
  
  We first bound the $L_{p_1,p_3}$ term. 
  We start by applying Lemma \ref{l:CLZ} to the triple $p_1,p_3,p_4$ to get (with some overestimating) that
  \begin{align*}
    d_{HS}(\pi_1(p_3),\pi_1(\overline{0,p_4}))^2 \leq C_0\Delta.
  \end{align*}
  Choose $u,v \in B_{\R^{n_1}}(1)$ so that $\pi_1(p_3) = u$ and $\pi_1(p_4) = v$.  Recall that we assumed that $\Delta < \mu$ for some yet to be determined $\mu$.  We now choose $\mu$ sufficiently small so that $\sqrt{C_0\mu} < 1$.
  Recalling that $d_{HS}$ on $G_1$ is equivalent to the Euclidean metric $|\cdot|$, we see that there is a $\gamma \in [0,1]$ so that $|u-\gamma v| \leq \sqrt{C_0\Delta} < 1$.  Applying Corollary \ref{l:close-lines} to $g = h = p_1$ with vectors $u$ and $\gamma v$, we get that there exists $C_1 > 0$ so that
  \begin{align}
    \sum_{i=1}^s \sup_{t \in [0,1]} \frac{d_{HS}(\pi_i(L_{0,p_3}(t)),\pi_i(L_{0,p_4}(\gamma t)))^{2i}}{\ell^{2i-1}} \leq C_1 \Delta. \label{e:p1p3-prebound}
  \end{align}
  As $\gamma \in [0,1]$, we get that $L_{0,p_4}(\gamma t)$ is contained in $\overline{0,p_4}$ and, remembering that $p_1 = 0$, we get
  \begin{align}
    \sum_{i=1}^s \sup_{t \in [0,1]} \frac{d_{HS}(\pi_i(L_{p_1,p_3}(t)),\pi_i(\overline{p_1,p_4}))^{2i}}{\ell^{2i-1}} \leq C_1 \Delta, \label{e:p1p3}
  \end{align}
  which takes care of the $L_{p_1,p_3}$ term in the summand.

  As for the $L_{p_3,p_4}$ term, we again seek to apply Corollary \ref{l:close-lines}, but we first need to verify that $p_3$ lies close to $(\gamma v,0,...,0) \in \overline{p_1,p_4}$.  To do so, we first take $t = 1$ in \eqref{e:p1p3-prebound} to get that $\tilde{\pi}(p_3) = L_{p_1,p_3}(1)$ is close to $(\gamma v,0,...,0)$ in the following way:
  \begin{align}
    \sum_{i=1}^s d_{HS}(\pi_i(\tilde{\pi}(p_3)), \pi_i(\gamma v, 0,...,0))^{2i} \leq C_1\Delta. \label{e:tildep-bound-1}
  \end{align}
  Here, we used the fact that $\pi_1(p_4) = v$ and the fact that $p_1 = 0$, $\ell=1$.  Applying Lemma \ref{l:curvature} (for many $i$) on the triple $p_1,p_2,p_3$, we get
  \begin{align*}
    \sum_{i=2}^s NH(\pi_i(p_3))^{2i} \leq sC_0\Delta.
  \end{align*}
  Using the definition of $NH$ and the fact that $\tilde{\pi}(\pi_i(p_3)) = \pi_i(\tilde{\pi}(p_3))$ for all $i$, we then get that $p_3$ is close to $\tilde{\pi}(p_3)$ in the following way:
  \begin{align}
    \sum_{i=1}^s d_{HS}(\pi_i(p_3), \pi_i(\tilde{\pi}(p_3)))^{2i} \leq s C_0 \Delta. \label{e:tildep-bound-2}
  \end{align}
  We now combine \eqref{e:tildep-bound-1} and \eqref{e:tildep-bound-2} to get that $p_3$ is close to $(\gamma v,0,...,0)$ in the following way:
  \begin{align*}
    \sum_{i=1}^s d_{HS}(\pi_i(p_3), \pi_i(\gamma v,0,...,0))^i \leq \sqrt{ s\sum_{i=1}^s d_{HS}(\pi_i(p_3), \pi_i(\gamma v, 0,...,0))^{2i}} \leq \sqrt{s 2^{2s-1} (sC_0 + C_1)\Delta}.
  \end{align*}
  Now an application of Lemma \ref{l:beta-balls} gives that $p_3 \in (\gamma v,0,...,0) \cdot B_{\R^n}(C_2\sqrt{\Delta})$ for some $C_2$ depending only on $G$.
  
  We can finally apply Corollary \ref{l:close-lines} with points $p_3$ and $(\gamma v,0,...,0)$ and vectors $v - u$ and $(1-\gamma)v$ as we did to get \eqref{e:p1p3-prebound} and \eqref{e:p1p3} to get that
  \begin{align}
    \sum_{i=1}^s \sup_{t \in [0,1]} \frac{d_{HS}(\pi_i(L_{p_3,p_4}(t)),\pi_i(\overline{p_1,p_4}))^{2i}}{\ell^{2i-1}} \leq C_3 \Delta \label{e:p3p4}
  \end{align}
  for some $C_3 > 0$.  The proposition now follows from \eqref{e:p1p3} and \eqref{e:p3p4}.
\end{proof}

Note that Proposition \ref{p:curvature} can only bound the drift between horizontal line segments that share some endpoint (namely either $p_1$ or $p_4$).  In the sequel, we will also be considering pairs of line segments that do not necessarily share an endpoint.  We now combine Lemma \ref{l:curvature} and Proposition \ref{p:curvature} to derive the following useful corollary, which allows us to bound horizontal line segments that do not share any endpoints.

\begin{corollary} \label{c:curvature}
  For any $\lambda > 0$, there exists a constant $C > 0$ depending only on $d_{HS}$, $\lambda$, and the Euclidean structure $|\cdot|$ so that the following holds.  Let $p_1,p_2,p_3,p_4,p_5 \in G$ be points so that $\lambda \ell \leq d_{HS}(p_i,p_j) \leq \ell$ if $i \neq j$ for some $\ell > 0$.  Then
  \begin{align*}
    \sum_{k=1}^s \sup_{t \in [0,1]} \frac{d_{HS}(\pi_k(L_{p_3,p_4}(t))),\pi_k(\overline{p_1,p_5}))^{2k}}{\ell^{2k-1}} \leq C \Delta,
  \end{align*}
  where $\Delta = \sum_{k=2}^s \Delta_k(p_1,p_2,p_3,p_4,p_5)$.
\end{corollary}

Like Proposition \ref{p:curvature}, the presence of $p_2$ on the right hand side cannot be removed without a stronger version of Lemma \ref{l:CLZ}.

\begin{proof}
  We dilate the setting so that $\ell = 1$.  By Lemma \ref{l:lines-bounds}, the left hand side of our desired inequality is then bounded by $s2^{2s}$.  Define
  \begin{align*}
    \Delta^{(1)} &= \sum_{k=2}^s \Delta_k(p_1,p_4,p_5), \\
    \Delta^{(2)} &= \sum_{k=2}^s \Delta_k(p_1,p_2,p_3,p_4).
  \end{align*}
  Then $\Delta^{(1)}, \Delta^{(2)} \geq 0$ and $\Delta^{(1)} + \Delta^{(2)} = \Delta$.  Let $C_0$ be the constant from Lemma \ref{l:curvature}.  If $C_0\Delta^{(1)} \geq 1$, then the corollary is satisfied by taking $C = s2^{2s}C_0$.  Thus, we may assume $C_0\Delta^{(1)} < 1$.

  We begin by showing $L_{p_1,p_4}$ is close to $\overline{p_1,p_5}$.  We first apply Lemma \ref{l:curvature} to $p_1,p_4,p_5$ and use the fact that $\ell=1$ to get
  \begin{align}
    d_{HS}(\pi_1(p_4),\pi_1(\overline{p_1,p_5}))^2  \leq C_0\Delta^{(1)}. \label{e:close1}
  \end{align}
  Let $u = \pi_1(p_1^{-1}p_4)$ and $v = \pi_1(p_1^{-1}p_5)$.  By \eqref{e:close1}, there is some $\gamma \in [0,1]$ so that $|u - \gamma v| \leq \sqrt{C_0 \Delta^{(1)}} < 1$.  We now apply Corollary \ref{l:close-lines} to $g = h = p_1$ and vectors $u,\gamma v$ with $\ell = 1$ so that
  \begin{align*}
    \max_{1 \leq k \leq s} \sup_{t \in [0,1]} d_{HS}(\pi_k(L_{p_1,p_4}(t)), \pi_k(\overline{p_1,p_5}))^{2k} \leq C_1 \Delta^{(1)} \leq C_1 \Delta,
  \end{align*}
  which gives
  \begin{align*}
    \sum_{k = 1}^s \sup_{t \in [0,1]} d_{HS}(\pi_k(L_{p_1,p_4}(t)), \pi_k(\overline{p_1,p_5}))^{2k} \leq s C_1 \Delta.
  \end{align*}
  
  We now show that $L_{p_3,p_4}$ is close to $\overline{p_1,p_4}$ by applying Proposition \ref{p:curvature} to $p_1,p_2,p_3,p_4$ to get
  \begin{align*}
    \sum_{k=1}^s \sup_{t \in [0,1]} d_{HS}(\pi_k(L_{p_3,p_4}(t))),\pi_k(\overline{p_1,p_4}))^{2k} \leq C_2 \Delta^{(2)} \leq C_2 \Delta.
  \end{align*}
  The last two inequalities, together with the triangle inequality and the inequality $(a+b)^p \leq 2^{p-1}(a^p + b^p)$ give the corollary.
\end{proof}

\subsection{Balls, filtrations, multiresolutions}
In this subsection, we will discretize the Carleson integral and reduce to a parametric setting.  Most of the setup will resemble sections 2.2-2.5 of \cite{li-schul-1} which go over the construction at much greater detail, including choices of parameters which still work in our setting.  We thus refer the reader to those sections for complete details.

In this section, we will endow $G$ with $d = d_\infty$.  Thus, all balls and diameters will be defined in terms of this metric.  The $d_\infty$ gives us a way to prove Lemma \ref{l:long-line} in Section \ref{s:flat} but is otherwise inconsequential.  We do not know if Lemma \ref{l:long-line} can be proven using $d_{HS}$.  Given a ball $B = B(x,r)$ and $\lambda > 0$, we define $\lambda B := B(x,\lambda r)$.

Let $\Gamma$ be a rectifiable curve in $G$.  Recall that we have $\diam(\Gamma) = 1$.  We may parameterize $\Gamma$ via a surjective 1-Lipschitz function $\gamma : \mathbb{T} \to \Gamma$ where $\mathbb{T}$ is a circle in $\R^2$ with circumference $C\Hd^1(\Gamma)$ where $C > 0$ depends on the metric $d$.  Indeed, this follows from the appendix of \cite{Schul-TSP} (see especially Lemma 3.7 of the appendix).  Note that although \cite{Schul-TSP} is proven for a geodesic metric on the Hilbert space, the proof also works for quasi-convex metrics on generic metric spaces by using \eqref{eq:qc} to connect points.  The Lipschitz bound on the parameterization will change depending on the quasi-convexity constant.

We fix a direction of flow along $\mathbb{T}$. By a subarc $\tau$ of $\gamma$, we mean a restriction of $\gamma$ onto a subinterval of $\mathbb{T}$.  We define set relations between two subarcs $\tau$ and $\zeta$ by using the set relation for their defining intervals.  However, the diameter of a subarc is the diameter of the {\it image}.

For each $n \in \N$, let $Y_n \subset \Gamma$ be a maximal $2^{-n}$ separated net\footnote{Recall an $\alpha$-separated net of $X$ is a subset $A \subset X$ for which $d(x,y) \geq \alpha$ for all $x,y \in A$.} and define the following multiresolution families of $\Gamma$
\begin{align*}
  \hat{\cG} &= \{B(x, 10 \cdot 2^{-n}) : x \in Y_n, n \in \Z\}, \\
  \cG &= \{B(x, 10 \cdot 2^{-n}) : x \in Y_n, n \in \N\}.
\end{align*}
Given a ball $B = B(x,r)$, we define $\hat{\beta}_\Gamma(B) := \hat{\beta}_\Gamma(x,r)$.

A straightforward modification of Lemma 2.6 of \cite{li-schul-1} shows that to prove Theorem \ref{th:necessity}, it suffices to prove
\begin{align*}
  \sum_{B \in \hat{\cG}} \hat{\beta}_\Gamma(B)^{2s} \diam(B) \leq C \Hd^1(\Gamma).
\end{align*}

\begin{lemma}
  There exists a constant $C > 0$ depending only on the metric structure of $G$ so that
  \begin{align*}
    \sum_{B \in \hat{\cG} \setminus \cG} \hat{\beta}_\Gamma(B)^{2s} \diam(B) \leq C \Hd^1(\Gamma).
  \end{align*}
\end{lemma}

\begin{proof}
  By picking any horizontal line $L$ that intersects $\Gamma$, we have for $B = B(x,10 \cdot 2^{-n}) \in \hat{\cG} \setminus \cG$ that
  \begin{align*}
    \hat{\beta}_\Gamma(B)^{2s} \leq \sum_{i=1}^s \left( \frac{\diam(\Gamma)}{10 \cdot 2^{-n}} \right)^{2i} \leq s2^{2n}.
  \end{align*}
  Recall that $\diam(\Gamma) = 1$ and $n \leq 0$.  As $G$ is a doubling metric space and $Y_n$ are $2^{-n}$-separated nets, a standard packing argument gives that
  \begin{align*}
    \sup_{n \leq 0} \# \{x \in Y_n : B(x,10 \cdot 2^{-n}) \cap \Gamma \neq \emptyset\} \leq M
  \end{align*}
  for some $M$ depending only on the metric structure of $G$.  Thus,
  \begin{align*}
    \sum_{B \in \hat{\cG} \setminus \cG} \hat{\beta}_\Gamma(B)^{2s} \diam(B) \leq \sum_{n \leq 0} M s2^{2n} \cdot 20 \cdot 2^{-n} \leq 40Ms \diam(\Gamma) \leq 40Ms \Hd^1(\Gamma).
  \end{align*}
\end{proof}

By the previous lemma, we see that it now suffices to prove that
\begin{align}
  \sum_{B \in \cG} \hat{\beta}_\Gamma(B)^{2s} \diam(B) \leq C \Hd^1(\Gamma). \label{e:hb-bound}
\end{align}
This is our goal for the remainder of the section.

Fix $J \in \N$.  We decompose $\cB := \{2B : B \in \cG\}$ into $D'$ families $\{\cB_i\}$ where for any two distinct balls $B_1,B_2 \in \cB_i$,
\begin{enumerate}
  \item $r(B_1)/r(B_2) \in 2^{J\Z}$,
  \item if $r(B_1) = r(B_2) = r$, then $d(B_1,B_2) > 3r$.
\end{enumerate}
In the second condition, we use the usual notion of distance between sets:
\begin{align*}
  d(B_1,B_2) = \inf_{x \in B_1, y \in B_2} d(x,y).
\end{align*}
By Lemma 2.11 of \cite{li-schul-1} (where we fix $C = 20$ and $\kappa = 3$), we have that $D'$ is bounded.
\begin{lemma}[Lemma 2.11 of \cite{li-schul-1}] \label{l:separate-scales}
  We may take $D' = D \cdot J$ where $D$ is some constant depending only on $G$.
\end{lemma}
We will do the decomposition with $J = 100$ so that $D'$ is a constant depending only on $G$.  Let $\cG^1 \cup ... \cup \cG^{D'}$ denote the decomposition of $\cG$ that follows from the decomposition $\cB_1 \cup ... \cup \cB_{D'}$.  We then apply Lemma 2.12 of \cite{li-schul-1} with $J$ and $\kappa = 3$ to each $\cB_i$, to produce a family of open ``dyadic cubes'' $\{\Delta(\cB,i)\}_{i=1}^{D'}$ so that
\begin{enumerate}
  \item for all $B \in \cG^i$, there exists $Q = Q(B) \in \Delta(\cB,i)$ for which $2B \subset Q \subset 2(1 + 2^{-J+2})B$,
  \item if $B,B' \in \cG^i$ is so that $r(B) < r(B')$ and $Q(B) \cap Q(B') \neq \emptyset$, then $Q(B') \supseteq Q(B)$,
  \item if $B,B' \in \cG^i$ is so that $r = r(B) = r(B')$ and $B \neq B'$, then $d(Q(B),Q(B')) > 2r$.
\end{enumerate}
Note that properties (1) and (3), together with the $2^J$ separation of scales of $\cG^i$, show that the map $B \mapsto Q(B)$ is injective.

For each $B \in \cG^i$, we define
\begin{align*}
  \Lambda(B) = \{\tau = \gamma|_I : &~I \subset \mathbb{T} \text{ is a connected component of } \gamma^{-1}(\Gamma \cap Q(B)) \text{ and } \gamma(I) \cap B \neq \emptyset\}.
\end{align*}
Thus, $\Lambda(B)$ is a collection of subarcs $\gamma|_I$.

Lemma 2.17 of \cite{li-schul-1} tells us that each $\cF^{0,i} = \bigcup_{B \in \cG^i} \Lambda(B)$ is a {\it prefiltration}, that is, for each $i$ there exists some $L_i > 0$ so that one has the decomposition $\cF^{0,i} = \bigcup_{n \geq 1} \cF^{0,i}_n$ into collections of arcs such that
\begin{enumerate}[label=(\arabic*)]
  \item For $\tau \in \cF_n^{0,i}$, $L_i2^{-nJ} \leq \diam(\tau) \leq L_i2^{-nJ+3}$,
  \item For distinct $\tau_1, \tau_2 \in \cF_n^{0,i}$, we must have $\tau_1 \cap \tau_2 = \emptyset$,
  \item If $\tau \in \cF_n^{0,i}$ and $\zeta \in \cF_{n+k}^{0,i}$ for $k > 0$, then either $\tau \cap \zeta = \emptyset$ or $\zeta \subset \tau$. 
\end{enumerate}
Recall that $J = 100$ is the same constant as before and set relation for arcs is defined by the set relation of their defining intervals.  We will not really deal with prefiltrations.  Instead, we use Lemma 2.13 of \cite{li-schul-1} to construct, for each prefiltration $\cF^0$ in $\{\cF^{0,i}\}_i$, a collection of arcs $\cF = \bigcup_n \cF_n$ so that
\begin{enumerate}
  \item For any $\tau \in \cF_n$, $L_i2^{-100n-10} \leq \diam(\tau) \leq L_i2^{-100n+4}$,
  \item For $\tau_1,\tau_2 \in \cF_n$, they are either disjoint, identical, or intersect in one or both of their endpoints,
  \item For $\zeta \in \cF_{n+1}$, there is a unique element $\tau \in \cF_n$ so that $\zeta \subset \tau$,
  \item For all $n$, $\bigcup_{\tau \in \cF_n} \tau = \mathbb{T}$,
  \item For each $\tau^0 \in \cF^0_n$, there is an element $\tau \in \cF_n$ so that $\tau \supset \tau^0$, and moreover if $I_0$ and $I$ are domains of $\tau^0$ and $\tau$, respectively, then the image under $\gamma$ of any connected component of $I \backslash I_0$ has diameter no more than $L_i2^{-100n-10}$.  Recall that $I, I_0$ are intervals,
  \item Different $\tau_0,\tau_1 \in \cF_n^0$ give rise to different arcs in $\cF_n$.
\end{enumerate}
By property (4), we are saying that the union of the domains of the subarcs is the entirety of $\mathbb{T}$.

Such collections of arcs are called {\it filtrations}, so we get $D'$ different filtrations.  Given a subarc $\tau$, we let $I_\tau \subseteq \mathbb{T}$ denote its parameterizing space.

To recap, we first construct our multiresolution family $\cG$ from the maximal $2^{-n}$-separated nets $Y_n$.  We then decomposed $\cG$ into $D'$ families of balls $\cG^i$, each with good scale separation.  For each $\cG^i$, we constructed families of dyadic cubes $\Delta(\cB,i)$ which we used to get our prefiltrations $\cF^{0,i}$.  Finally, we used each $\cF^{0,i}$ to construct a family of filtrations $\cF^i$.

As property (5) of filtrations tells us that each $\tau \in \cF^{0,i}$ gives rise to some $\tau' \in \cF^i$ of the same ``generation'', we define
\begin{align*}
  \Lambda'(B) = \{\tau' : \tau \in \Lambda(B)\}.
\end{align*}

We record the following lemma which controls the diameter of arcs in filtrations relative to the ball they come from.

\begin{lemma}
  If $\tau \in \Lambda'(B)$, then
  \begin{align}
    \frac{1}{2} \diam(B) \leq \diam(\tau) \leq \diam(3B) = 3\diam(B). \label{e:curve-diam-bound}
  \end{align}
\end{lemma}

\begin{proof}
  As $\lambda B = x\delta_\lambda(x^{-1}B)$ where $x$ is the center of $B$, we get that $\diam(\lambda B) = \lambda \diam(B)$ when $\diam$ is computed with any metric that scales with dilation.

  Let $\tau^0 \in \Lambda(B)$ be the subarc that gave rise to $\tau$ so that $\tau \supseteq \tau^0$.  We have that
  \begin{align*}
    \frac{1}{2} \diam(B) \leq \diam(\tau^0) \leq 2(1+2^{-98})\diam(B).
  \end{align*}
  The upper bound comes from the fact that the image of $\tau^0$ lies in $Q(B)$ and property (1) of cubes.  The lower bound comes from the fact that $\tau^0$ goes from inside $B$ to the boundary of $Q(B) \supset 2B$.  The lower bound on $\diam(\tau)$ now follows from the fact that $\tau$ contains $\tau^0$.

  We now derive the upper bound.  Let $B \in \cG^i$.  We first note that property (1) of prefiltrations and our upper bound on $\tau^0$ gives for some minimal $n$ that
  \begin{align*}
    L_i2^{-100n} \leq \diam(\tau^0) \leq 2(1+2^{-98}) \diam(B)
  \end{align*}
  Now let $I,I^0$ be the domains of $\tau$ and $\tau^0$, and let $I'$ be a connected interval of $I \setminus I^0$.  By property (5) of filtrations, we have that
  \begin{align*}
    \diam(\tau|_{I'}) \leq L_i2^{-100n-10} < \frac{1}{4} \diam(B).
  \end{align*}
  Finally, as the image of $\tau|_{I^0} = \tau^0$ lies in $Q(B) \subseteq 2(1+2^{-98}) B$, we get the upper bound of $\diam(3B)$ on $\diam(\tau)$.
\end{proof}

Given a curve $\tau$, we let $a(\tau)$ and $b(\tau)$ denote the endpoints of the domain where $a(\tau) < b(\tau)$.  We also define
\begin{align*}
  L_\tau := \overline{\gamma(a(\tau)),\gamma(b(\tau))},
\end{align*}
and
\begin{align*}
  \beta_i(\tau) &:= \sup_{t \in I_\tau} \frac{d(\pi_i(\gamma(t)),\pi_i(L_\tau))}{\diam(\tau)}, \quad i \in \{1,...,s\}, \\
  \hat{\beta}(\tau)^{2s} &:= \sum_{i=1}^s \beta_i(\tau)^{2i}.
\end{align*}
We define
\begin{align*}
  \cG_2 &:= \{B \in \cG : \hat{\beta}(\tau) < \epsilon_0 \hat{\beta}_\Gamma(B) ~\forall \tau \in \Lambda'(B)\}, \\
  \cG_1 &:= \cG \backslash \cG_2.
\end{align*}
Here $\epsilon_0 > 0$ will be a sufficiently small constant that we choose in subsection \ref{s:flat}.

Balls of $\cG_1$ have one curve whose ``wigglyness'' captures most of $\hat{\beta}_\Gamma(B)$ whereas curves of $\cG_2$ are mostly flat, but it is the combination of two or more of these curves that gives $\hat{\beta}_\Gamma(B)$.  See Figure \ref{g:balls} for representative images.  We call the families $\cG_1$ and $\cG_2$ non-flat and flat balls, respectively.
\begin{figure} 
  \includegraphics[scale=0.3]{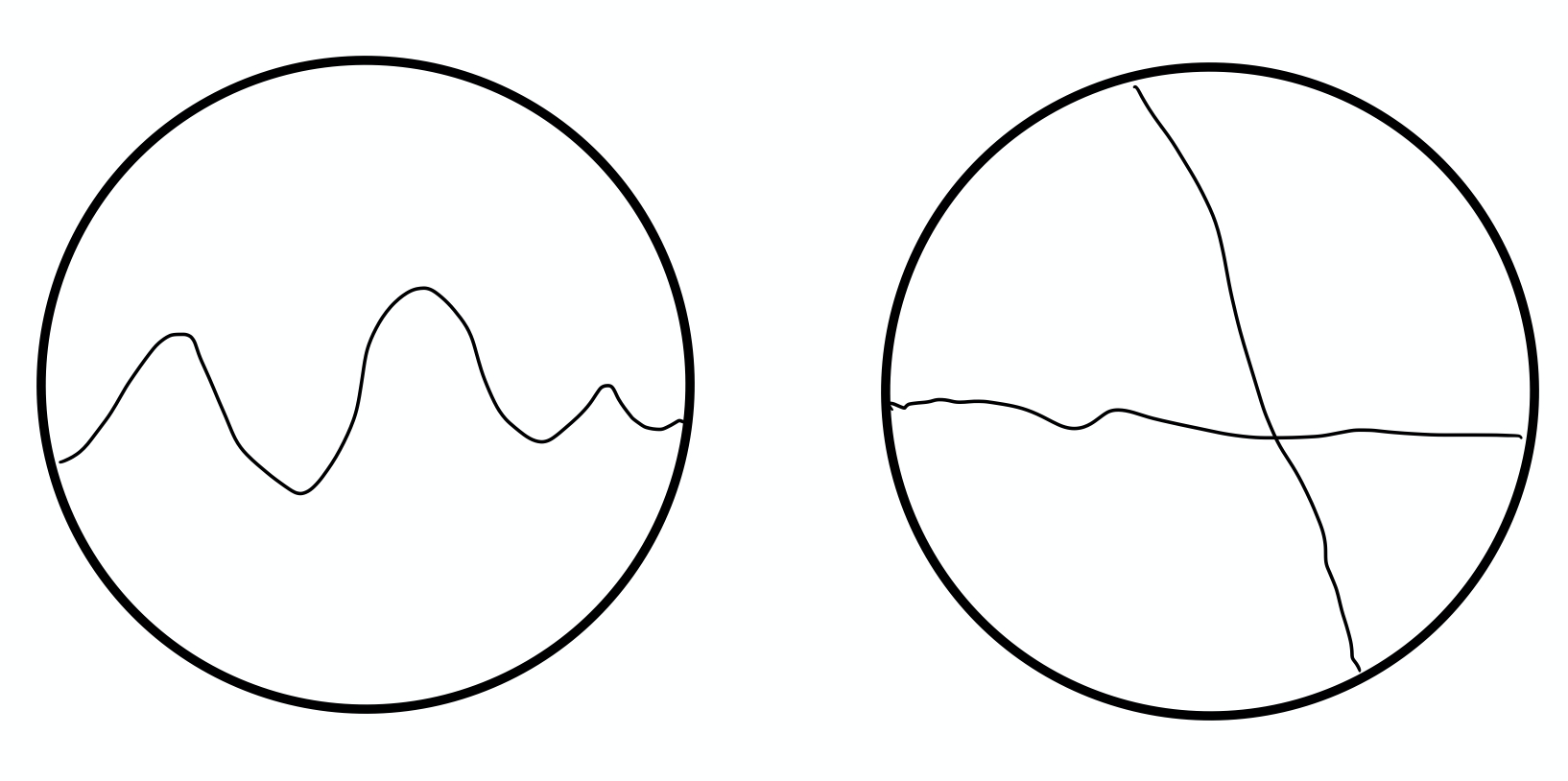}
  \caption{Balls of type $\cG_1$ and $\cG_2$.  In the second ball, the curves should be close to horizontal lines.}
  \label{g:balls}
\end{figure}

To prove \eqref{e:hb-bound}, it now suffices to prove
\begin{align}
  \sum_{B \in \cG_1} \hat{\beta}_\Gamma(B)^{2s} \diam(B) &\leq C\Hd^1(\Gamma), \label{e:nonflat} \\
  \sum_{B \in \cG_2} \hat{\beta}_\Gamma(B)^{2s} \diam(B) &\leq C\Hd^1(\Gamma). \label{e:flat}
\end{align}
We will prove \eqref{e:nonflat} in the following subsection and \eqref{e:flat} in the subsection after.

\subsection{Non-flat balls} \label{s:nonflat}
The main goal of this subsection will be to prove the following proposition:
\begin{proposition} \label{p:tb-bound}
  There exists a $C > 0$ depending only on $d$ so that for any filtration $\cF$ and $i \in \{1,...,s\}$, we have
  \begin{align*}
    \sum_{\tau \in \cF} \beta_i(\tau)^{2i} \diam(\tau) \leq C\Hd^1(\Gamma).
  \end{align*}
\end{proposition}

Assuming this proposition, one can now derive \eqref{e:nonflat}.
\begin{corollary}
  \begin{align*}
    \sum_{B \in \cG_1} \hat{\beta}_\Gamma(B)^{2s} \diam(B) \leq C \Hd^1(\Gamma).
  \end{align*}
\end{corollary}

\begin{proof}
  For each $B \in \cG_1$, let $\tau_B \in \Lambda'(B)$ denote a curve so that $\hat{\beta}(\tau_B) \geq \epsilon_0 \hat{\beta}_\Gamma(B)$.  Recall that $\cG$ was decomposed into $D'$ families $\{\cG^i\}_{i=1}^{D'}$, which thus also partitions $\cG_1$ into $\{\cG_1^i\}_{i=1}^{D'}$.  Let $\cF_i$ denote the filtration constructed using $\cG_1^i$.  By \eqref{e:curve-diam-bound}, $\diam(\tau_B) \geq \frac{1}{2}\diam(B)$.  As $D'$ is a constant depending only on $G$, we now get
  \begin{align*}
    \sum_{B \in \cG_1} \hat{\beta}_\Gamma(B)^{2s} \diam(B) &= \sum_{i=1}^{D'} \sum_{B \in \cG_1^i} \hat{\beta}_\Gamma(B)^{2s} \diam(B) \leq \frac{2}{\epsilon_0^{2s}} \sum_{i=1}^{D'} \sum_{B \in \cG_1^i} \hat{\beta}(\tau_B)^{2s} \diam(\tau_B) \\
    &\leq \frac{2}{\epsilon_0^{2s}} \sum_{i=1}^{D'} \sum_{\tau \in \cF_i} \hat{\beta}(\tau)^{2s} \diam(\tau) = \frac{2}{\epsilon_0^{2s}} \sum_{i=1}^{D'} \sum_{\tau \in \cF_i} \sum_{j=1}^s \beta_j(\tau)^{2j} \diam(\tau) \leq C\Hd^1(\Gamma).
  \end{align*}
\end{proof}

We now start working on Proposition \ref{p:tb-bound}.  We will now view $i \in \{1,...,s\}$ and the filtration $\cF$ as fixed.  For simplicity, we will assume for this filtration $\cF$ that $L=1$ in property (1) of filtrations.  Indeed, one could just first scale everything by $L^{-1}$, prove the results of this subsection, and then rescale back.

We define for $\tau \in \cF_n$ and $k \in \N$ the families
\begin{align*}
  \cF_{\tau,k} := \{\tau' \in \cF_{n+k} : \tau' \subset \tau\}.
\end{align*}
Given $\tau \in \cF$, we further define
\begin{align*}
  d_\tau := \max_{\tau' \in \cF_{\tau,1}} \sup_{z \in \pi_i(L_{\tau'})} d(z,\pi_i(L_\tau)).
\end{align*}

\begin{lemma} \label{l:dtau-bound}
  \begin{align}
    \sum_{\tau \in \cF} \frac{d_\tau^{2i}}{\diam(\tau)^{2i-1}} \leq C\Hd^1(\Gamma). \label{e:dtau-bound}
  \end{align}
\end{lemma}

\begin{proof}
  Although the quantities $d_\tau$, $\diam$, and $\Hd^1$ are all calculated with respect to $d = d_\infty$, it suffices to prove \eqref{e:dtau-bound} with $d = d_{HS}$ instead.  Indeed, a biLipschitz change of metric affects all quantities by only a multiplicative constant.  Thus, for the rest of this proof, we will assume $\diam$, $d_\tau$, and $\Hd^1$ are with respect to $d_{HS}$.

  As $d_{HS}$ and $d$ are biLipschitz, property (1) of filtrations gives that there exists a $C_0 > 1$ so that
  \begin{align}
    \frac{1}{C_0} 2^{-100k} \leq \diam(\tau) \leq C_0 2^{-100k}, \qquad \forall \tau \in \cF_k. \label{e:HS-diams}
  \end{align}
  Define $M = \lceil 3 + \log_2 C_0 \rceil$.  We have chosen $M$ so that for all $k$
  \begin{align}
    C_02^{-100(k+M)} \leq C_0^{-1} 2^{-202-100k}. \label{e:C0-defn}
  \end{align}
  For $\ell \in \{2,...,s\}$, define
  \begin{align*}
    \Delta_\ell(\tau) := \left( \sum_{\tau' \in \cF_{\tau,M}} d_{HS}(\pi_\ell(\gamma(a(\tau'))),\pi_\ell(\gamma(b(\tau')))) \right) - d_{HS}(\pi_\ell(\gamma(a(\tau))),\pi_\ell(\gamma(b(\tau)))).
  \end{align*}
  It suffices to prove that
  \begin{align}
    \frac{d_\tau^{2i}}{\diam(\tau)^{2i-1}} \leq C\sum_{\ell=2}^s \Delta_\ell(\tau). \label{e:delta-bound}
  \end{align}
  Indeed, summing both sides over $\tau \in \cF$ gives
  \begin{align*}
    \sum_{\tau \in \cF} \frac{d_\tau^{2i}}{\diam(\tau)^{2i-1}} \leq C \sum_{\ell=2}^s \sum_{\tau \in \cF} \Delta_\ell(\tau) \leq sCM \Hd^1(\Gamma)
  \end{align*}
  where we used the fact that $\sum_\tau \Delta_\ell(\tau) \leq M\Hd^1(\Gamma)$ for each $\ell$ as $\pi_\ell$ is 1-Lipschitz and the summation of differences telescopes.

  Fix a $\tau \in \cF_k$ and let $\{\tau_j\}_{j=1}^m$ denote the subarcs of $\cF_{\tau,1}$ in order of the flow along $\bT$ so that $a(\tau_1) = a(\tau)$ and $b(\tau_m) = b(\tau)$. Let $P$ denote the images under $\gamma$ of all the endpoints of subarcs of $\cF_{\tau,2}$.  We will want to apply the results of the last subsection to points from $P$.  It may seem strange that we are taking the endpoints of $\cF_{\tau,2}$ instead of endpoints of $\cF_{\tau,1}$, but this is necessary as we will indicate below.

  Let us first assume $P$ is a $C_0^{-1}2^{-1000-100k}$-separated set (in $d_{HS}$).  Note that we also have
  \begin{align}
    d_{HS}(p,q) \leq \diam(\tau), \qquad \forall p,q \in P. \label{e:P-upperbound}
  \end{align}
  We will bound
  $$\sup_{z \in \pi_i(L_{\tau_j})} \frac{d_{HS}(z,\pi_i(L_\tau))^{2i}}{\diam(\tau)^{2i-1}}$$
  for different $j$s and take the maximal of these bounds to get our needed result.  There will be three cases, $j = 1$, $j = m$, and the rest.

  First we consider when $j \in \{2,...,m-1\}$.  Pick $u$, any endpoint of a subarc of $\cF_{\tau_1,1}$ that is not also an endpoint of $\tau_1$.  It then follows that $u$ lies in $P$.  Note that it is here that we used the fact that $P$ contains endpoints of $\cF_{\tau,2}$.  We apply Corollary \ref{c:curvature} to the collection of points $\gamma(a(\tau)), \gamma(u), \gamma(a(\tau_j)), \gamma(b(\tau_j)),\gamma(b(\tau)) \in P$ with $\ell = \diam(\tau) \leq C_0 2^{-100k}$ and $\lambda = C_0^{-2} 2^{-1000}$ to get that
  \begin{align}
    \sup_{z \in \pi_i(L_{\tau_j})} \frac{d_{HS}(z,\pi_i(L_\tau))^{2i}}{\diam(\tau)^{2i-1}} \leq C\sum_{\ell=2}^s \Delta_\ell(\tau) \label{e:middle}
  \end{align}
  where $u$ plays the role of $p_2$ in Corollary \ref{c:curvature}.

  For $j = 1$, we apply Proposition \ref{p:curvature} to the points $\gamma(a(\tau)), \gamma(u), \gamma(b(\tau_1)), \gamma(b(\tau))$ using the same $u$, $\ell$, $\lambda$ as before to get that
  \begin{align}
    \sup_{z \in \pi_i(L_{\tau_1})} \frac{d_{HS}(z,\pi_i(L_\tau))^{2i}}{\diam(\tau)^{2i-1}} \leq C\sum_{\ell=2}^s \Delta_\ell(\tau). \label{e:first}
  \end{align}
  
  For $j = m$, we apply Proposition \ref{p:curvature} to the points $\gamma(a(\tau)), \gamma(a(\tau_{m-1})),\gamma(a(\tau_m)), \gamma(b(\tau_m))$ to get that
  \begin{align}
    \sup_{z \in \pi_i(L_{\tau_m})} \frac{d_{HS}(z,\pi_i(L_\tau))^{2i}}{\diam(\tau)^{2i-1}} \leq C\sum_{\ell=2}^s \Delta_\ell(\tau). \label{e:last}
  \end{align}
  The inequalities \eqref{e:middle}, \eqref{e:first}, \eqref{e:last} give \eqref{e:delta-bound} in our current case where $P$ is $C_0^{-1}2^{-1000-100k}$-separated.

  Thus, we now assume that $P$ is not $C_0^{-1}2^{-1000-100k}$-separated, and let $u,v$ be the endpoint of two subarcs of $\cF_{\tau,2}$ (where $u < v$ based on the flow of $\bT$) so that $d_{HS}(\gamma(u),\gamma(v)) < C_0^{-1}2^{-1000-100k}$.  Let $\xi$ be any subarc of $\cF_{\tau,2} \subset \cF_{k+2}$ that lies between $u$ and $v$.  As $\diam(\xi) \geq C_0^{-1} 2^{-100(k+2)}$, there then exists some $w \in \xi$ so that
  \begin{align*}
    d_{HS}(\gamma(w),\{\gamma(u),\gamma(v)\}) \geq \diam(\xi) - d_{HS}(\gamma(u),\gamma(v)) \geq C_0^{-1} 2^{-201-100k}.
  \end{align*}
  This $w$ must lie in some subarc $\zeta \in \cF_{\xi,M-2} \subset \cF_{\tau,M}$ and as $\diam(\zeta) \leq C_02^{-100(k+M)}$, we get that if $z$ is an endpoint of $\zeta$, then
  \begin{multline*}
    d_{HS}(\gamma(z), \{\gamma(u),\gamma(v)\}) \geq d_{HS}(\gamma(w), \{\gamma(u),\gamma(v)\}) - \diam(\zeta) \\
    \geq C_0^{-1}2^{-201-100k} - C_02^{-100(k+M)} \overset{\eqref{e:C0-defn}}{\geq} C_0^{-1}2^{-202-100k}.
  \end{multline*}
  Thus,
  $$d_{HS}(\gamma(u),\gamma(z)) + d_{HS}(\gamma(z),\gamma(v)) - d_{HS}(\gamma(u),\gamma(v)) > C_0^{-1}2^{-203-100k}.$$
  As $z$ is an endpoint of a subarc of $\cF_{\tau,M}$ and lies between $u,v$ which are endpoints of subarcs of $\cF_{\tau,2}$, we get that by a simple triangle inequality argument that
  \begin{align*}
    C_0^{-1}2^{-203-100k} \leq d_{HS}(\gamma(u),\gamma(z)) + d_{HS}(\gamma(z),\gamma(v)) - d_{HS}(\gamma(u),\gamma(v)) \leq \Delta_s(\tau).
  \end{align*}
  Finally, Lemma \ref{l:lines-bounds} and \eqref{e:P-upperbound} gives us that $\frac{d_\tau^{2i}}{\diam(\tau)^{2i-1}} \leq 2^{2i}\diam(\tau) \leq C_02^{2s-100k}$ which, with the previous inequality, gives our result.
\end{proof}

Given $\tau \in \cF$, we define $\{\tau_j\}_{j = 0}^\infty$ inductively as $\tau_0 = \tau$ and $\tau_j \in \cF_{\tau,j}$ so that $d_{\tau_j}$ is maximal for all arcs in $\cF_{\tau,j}$.  We now show that the $d_{\tau_k}$ controls the $\beta_i(\tau)$.  The proof of the following lemma is exactly like Lemma 3.6 of \cite{li-schul-1}.
\begin{lemma}[Lemma 3.6 of \cite{li-schul-1}]
  Let $\tau \in \cF$.  Then
  \begin{align}
    \beta_i(\tau) \diam(\tau) \leq \sum_{k=0}^\infty d_{\tau_k}. \label{e:tb-dk}
  \end{align}
\end{lemma}

\begin{proof}[Proof of Proposition \ref{p:tb-bound}]
  We compute in $\ell_{2i}$-fashion:
  \begin{align*}
    \left( \sum_{\tau \in \cF} \beta_i(\tau)^{2i} \diam(\tau) \right)^{1/2i} &\overset{\eqref{e:tb-dk}}{\leq} \left( \sum_{\tau \in \cF} \frac{\left( \sum_{k=0}^\infty d_{\tau_k} \right)^{2i}}{\diam(\tau)^{2i-1}} \right)^{1/2i} \\
    &\leq \sum_{k=0}^\infty \left( \sum_{\tau \in \cF} \frac{d_{\tau_k}^{2i}}{\diam(\tau)^{2i-1}} \right)^{1/2i} \\
    &\leq \sum_{k=0}^\infty 2^{-J(k-1)/2} \left( \sum_{\tau \in \cF} \frac{d_{\tau_k}^{2i}}{2^{-J(k-1)(2i-1)}\diam(\tau)^{2i-1}} \right)^{1/2i} \\
    &\leq \sum_{k=0}^\infty 2^{-J(k-1)/2} \left( \sum_{\tau \in \cF} \frac{d_{\tau_k}^{2i}}{\diam(\tau_k)^{2i-1}} \right)^{1/2i} \\
    &\overset{\eqref{e:dtau-bound}}{\leq} \sum_{k=0}^\infty 2^{-J(k-1)/2} C\Hd^1(\Gamma)^{1/2i} \\
    &\leq C\Hd^1(\Gamma)^{1/2i}.
  \end{align*}
\end{proof}

\subsection{Flat balls} \label{s:flat}
We are back to using just the $d_\infty$ metric. One can compare the following lemma with Lemma 4.3 of \cite{li-schul-1}.  The crucial difference is that we are setting $h = \beta_1(\tau')\diam(\tau')$ rather than $\beta_s(\tau')\diam(\tau')$.

\begin{lemma} \label{l:long-line}
  There exists $\eta_0 \in (0,1)$ so that the following holds.  Let $B \in \cG_2$ be a ball of radius $r$ and $Q = Q(B)$ for which $2B \subset Q \subset 2(1+2^{-98})B$.  If $\tau' \in \Lambda'(B)$ is such that $\Center(B) \in \tau'$, $\beta_s(\tau') < \eta_0$, and
  \begin{align*}
    h := \beta_1(\tau')\diam(\tau') < \frac{1}{10}r,
  \end{align*}
  then there is a subarc $\tilde{\tau} \subset \tau'$ with image in $2B$ such that $\diam(\tilde{\tau}) \geq 4r-20h$.
\end{lemma}

\begin{proof}
  We translate the setting so that $\Center(B) = 0$ and $r = 1$.  Let $k = \beta_s(\tau')\diam(\tau')$, $L = L_{\tau'}$ and $C(L,k) = \{ p \in G : d(p,L) \leq k\}$.  Then $\tau' \subset C(L,k)$. 
  As $0 \in \tau' \subset C(L,k)$, one gets that $d(0,L) \leq k$.  Recall that horizontal lines going through 0 are lines of $\R^{n_1} \times \{0\}$ that go through the origin.  It thus follows from continuity of group multiplication that if $\eta_0$ (and so $k$ also) is sufficiently small, then $L$ is close in Hausdorff distance to a line of $\R^{n_1} \times \{0\}$ and so
  \begin{align}
    C(L,k) \cap 2B = C(L,k) \cap (B_{\R^{n_1}}(2) \times \R^{n_2 + ... + n_s}). \label{e:product-tube}
  \end{align}
  Here, we have used the fact that $B = \{x : N_\infty(x) \leq 1\}$ has a product structure.  This is where we have used the fact that our metric is $d_\infty$.  We will not use it anywhere else.
  
  As $\tau' \subset C(L,k)$, we get from \eqref{e:product-tube} that
  \begin{align*}
    \pi_1(\tau' \cap 2B) = \pi_1(\tau' \cap (B_{\R^{n_1}}(2) \times \R^{n_2 + ... + n_s})) = \pi_1(\tau') \cap B_{\R^{n_1}}(2).
  \end{align*}
  Let $\xi' = \pi_1(\tau')$, which is now a curve in $\pi_1(G)$.  As $\pi_1 : G \to \pi_1(G)$ is 1-Lipschitz, it suffices then to find a subcurve $\tilde{\xi} \subset \xi'$ contained in $B_{\pi_1(G)}(2)$ so that $\diam(\tilde{\xi}) \geq 4 - 20h$.
  
  Recall that the metric $d_{HS}$ on $\pi_1(G) \cong \R^{n_1}$ is equivalent to the Euclidean metric $|\cdot|$.  Thus, we have now reduced to the following problem in $\R^{n_1}$: We have a cylinder $\pi_1(C(L,h))$ of radius $h$ containing both 0 and a curve $\xi'$.  Because the central axis of the cylinder $\pi_1(L)$ goes from $\xi'(a(\tau')) \in B_{\R^{n_1}}(2)^c$ to $\xi'(b(\tau')) \in B_{\R^{n_1}}(2)^c$ but the cylinder also contains 0, these endpoints must lie at opposite ends of the cylinder and so are almost antipodal with respect to the ball.  We now want to find a subcurve $\tilde{\xi} \subset \xi'$ contained in $B_{\R^{n_1}}(2)$ with $\diam(\tilde{\xi}) > 4 - 20h$.  The existence of such a subcurve is a simple (but tedious) Euclidean exercise that involves finding the maximum subcurve of $\xi'$ in $B_{\R^{n_1}}(2)$.  We leave the details for the reader.
\end{proof}

We now choose once and for all $\epsilon_0 < \min\left\{\frac{\eta_0}{1000s},\frac{1}{6000s^{-1/2s}}\right\}$.  Note that if $\tau' \in \Lambda'(B)$ is such that $B \in \cG_2$ and $\tau' \ni \text{Center}(B)$, then the hypothesis of Lemma \ref{l:long-line} is automatically satified as
\begin{align}
  \beta_s(\tau') \leq \hat{\beta}(\tau') < \epsilon_0 \hat{\beta}_\Gamma(B) < \frac{\eta_0}{1000} < \frac{1}{1000}. \label{e:betas-bounds}
\end{align}

The following lemma will allow us to use the definition of flat balls to find a large subcurve $\hat{\xi}$ that is far away from $\tau'$.

\begin{lemma} \label{l:disjoint-line}
  Let $B \in \cG_2$ be a ball of radius $r$.  Let $\tau' \in \Lambda'(B)$ be such that $\Center(B) \in \tau'$.  Then there is a $\xi \in \Lambda(B)$ with a subarc $\hat{\xi} \subset \xi$ with image inside $2B$ of diameter
  \begin{align*}
    \diam(\hat\xi) > 100 \epsilon_0^s \hat{\beta}_\Gamma(B)^s \diam(B)
  \end{align*}
  so that
  \begin{align*}
    d(\hat\xi,\tau') > 100 \epsilon_0^s \hat{\beta}_\Gamma(B)^s \diam(B).
  \end{align*}
\end{lemma}

\begin{proof}
  Define $L = L_{\tau'}$.  By definition, we have
  \begin{align*}
    \sum_{i=1}^s \sup_{x \in B \cap \Gamma} \left( \frac{d(\pi_i(x),\pi_i(L))}{r} \right)^{2i} \geq \hat{\beta}_\Gamma(B)^{2s}.
  \end{align*}
  Thus, there must exist an $i$ and $x \in B \cap \Gamma$ for which
  \begin{align*}
    \left( \frac{d(\pi_i(x),\pi_i(L))}{r} \right)^{2i} \geq \frac{\hat{\beta}_\Gamma(B)^{2s}}{2s}.
  \end{align*}
  We can now bound
  \begin{multline*}
    \epsilon_0 \overset{\eqref{e:curve-diam-bound}}{<} \frac{\diam(B)}{2000s^{-1/2s}\diam(\tau')} = \frac{\diam(B)/2\diam(\tau')}{1000s^{-1/2s}} \overset{\eqref{e:curve-diam-bound}}{\leq} \frac{(\diam(B)/2\diam(\tau'))^{i/s}}{1000s^{-1/2s}} \\
    \leq \frac{(r/\diam(\tau'))^{i/s}}{1000s^{-1/2s}}
  \end{multline*}
  and so
  \begin{align*}
    \left( \frac{d(\pi_i(x),\pi_i(L))}{\diam(\tau')} \right)^{2i} \geq (1000\epsilon_0)^{2s} s\left( \frac{d(\pi_i(x),\pi_i(L))}{r} \right)^{2i} \geq (500\epsilon_0 \hat{\beta}_\Gamma(B))^{2s} > \hat{\beta}(\tau')^{2s}.
  \end{align*}
  Thus, $x \notin \tau'$ and so $x \in \xi$ for some other $\xi \in \Lambda(B)$ for which
  \begin{align*}
    d(x,L) \geq d(\pi_i(x),\pi_i(L)) \geq 500 (\epsilon_0 \hat{\beta}_\Gamma(B))^{s/i} \diam(\tau').
  \end{align*}
  On the other hand, we have that
  \begin{align*}
    \sup_{t \in I_{\tau'}} d(\pi_i(\gamma(t)),\pi_i(L)) \leq \hat{\beta}(\tau')^{s/i} \diam(\tau') < (\epsilon_0 \hat{\beta}_\Gamma(B))^{s/i} \diam(\tau').
  \end{align*}
  Altogether we get that
  \begin{align*}
    d(x,\tau') \geq d(\pi_i(x),\pi_i(\tau')) > 499 (\epsilon_0 \hat{\beta}_\Gamma(B))^{s/i} \diam(\tau') \overset{\eqref{e:betas-bounds}}{\geq} 499(\epsilon_0 \hat{\beta}_\Gamma(B))^s \diam(\tau').
  \end{align*}
  The lemma now follows by taking $\hat{\xi}$ to be a maximal subarc of $\xi$ in $B(x,100\epsilon_0^s \beta_\Gamma(B)^s\diam(B))$ containing $x$ and remembering that $x \in B$ and $100\epsilon_0^s\beta_\Gamma(B)^s \leq 1/2$.
\end{proof}

\begin{proposition} 
  Let $B \in \cG_2$ and $E = \Gamma \cap 2B$.  If we cover $E$ with balls $\{B_i\}$ so that $\diam(B_i) < 10\epsilon_0^s \hat{\beta}_\Gamma(B)^s \diam(B)$, then
  \begin{align}
    \sum_i \diam(B_i) \geq [2 + (\epsilon_0 \hat{\beta}_\Gamma(B))^s] \diam(B). \label{e:ball-cover}
  \end{align}
\end{proposition}

\begin{proof}
  Let $\tau' \in \Lambda'(B)$ contain $\Center(B)$, $\tilde{\tau} \subset \tau'$ be the subcurve from Lemma \ref{l:long-line}, and $\hat{\xi}$ be the curve from Lemma \ref{l:disjoint-line} (applied to $\tau'$).  Note that any $B_i$ may intersect at most one of the images of $\tilde{\tau}$ and $\hat{\xi}$.  Thus, as $\beta_1(\tau') \leq \hat{\beta}(\tau')^s \leq \epsilon_0^s \hat{\beta}_\Gamma(B)^s$, we get that
  \begin{align*}
    \sum_i \diam(B_i) &\geq \sum_{B_i \cap \tilde{\tau} \neq \emptyset} \diam(B_i) + \sum_{B_i \cap \hat{\xi} \neq \emptyset} \diam(B_i) \geq \diam(\tilde{\tau}) + \diam(\hat{\xi}) \\
    &\geq 2\diam(B) - 20 \epsilon_0^s \hat{\beta}_\Gamma(B)^s \diam(\tau') + 100 \epsilon_0^s \hat{\beta}_\Gamma(B)^s \diam(B) \\
    &\overset{\eqref{e:curve-diam-bound}}{\geq} (2 + \epsilon_0^s \hat{\beta}_\Gamma(B)^s) \diam(B).
  \end{align*}
\end{proof}

Now fix an integer $M \geq \lfloor -\log_2(s) \rfloor =: M_0$.  As $\hat{\beta}_E \leq s$ always, we can then define
\begin{align*}
  \cB^M := \{B \in \cG_2 : \hat{\beta}_\Gamma(B) \in [2^{-M-1},2^{-M}]\}.
\end{align*}
Set $J_M = \lceil s(M - M_0) - s\log(10\epsilon_0) + 10 \rceil$ and we apply Lemma \ref{l:separate-scales} to $\cB^M$ with $J = J_M$ to get $\{\cB_i^M\}_{i=1}^{D_M}$ where $D_M = D \cdot J_M$.  Note that this is a second application of Lemma \ref{l:separate-scales} with a new value of $J$ (as opposed to $J = 100$).  For each $i \in \{1,...,D_M\}$, apply the dyadic cube construction following Lemma \ref{l:separate-scales} to $\cB_i^M$ to get dyadic cubes $\Delta(\cB^M,i)$.  Fix one such $\Delta = \Delta(\cB^M,i)$.

\begin{proposition} \label{p:cube-cover}
  Let $B \in \cG_2$ and suppose $Q = Q(B) \in \Delta$ is decomposed as
  \begin{align*}
    Q = \left( \bigcup_i Q_i \right) \cup R_Q
  \end{align*}
  where $Q_i = Q(B_i) \in \Delta$ are maximal so that $Q_i \subsetneq Q$ and $R_Q$ is chosen so that the union is disjoint.  Then
  \begin{align*}
    \sum_i \diam(Q_i) + \Hd^1(R_Q \cap \Gamma) \geq \diam(Q) \left( 1 + \frac{1}{10} \epsilon_0^s \beta_\Gamma(B)^s \right).
  \end{align*}
\end{proposition}

\begin{proof}
  The proof is essentially the same as Proposition 4.7 of \cite{li-schul-1} except we will use our \eqref{e:ball-cover} instead of their equation (28).
\end{proof}

We can now prove the following proposition.
\begin{proposition}
  \begin{align*}
    \sum_{Q \in \Delta} \diam(Q) \leq \frac{10}{\epsilon_0^s} 2^{sM} \Hd^1(\Gamma).
  \end{align*}
\end{proposition}

\begin{proof}
  The proof is essentially the same as Proposition 4.8 of \cite{li-schul-1} except we use our Proposition \ref{p:cube-cover} instead of their Proposition 4.7.  Note that this changes their quantity $q$ to $q = (1 + \epsilon_0^s 2^{-sM}/10)^{-1}$ and so we get
  \begin{align*}
    \sum q^n = \frac{1}{1-q} \leq \frac{10}{\epsilon_0^s} 2^{-sM}.
  \end{align*}
\end{proof}

We can finally finish by proving the following proposition, which is stronger than \eqref{e:flat}, the goal of this subsection.

\begin{proposition} \label{e:flat-bound}
  \begin{align*}
    \sum_{B \in \cG_2} \hat{\beta}_\Gamma(B)^{s+1} \diam(B) \leq C \Hd^1(\Gamma).
  \end{align*}
\end{proposition}

\begin{proof}
  The proof is essentially the same as the proof of equation (13) of \cite{li-schul-1} (right after the proof of their Proposition 4.8).  Instead of raising by 2, we raise by $s+1$.  See also the discussion after the proof of Lemma 4.7 of \cite{CLZ} for a small correction of the argument in \cite{li-schul-1}.
\end{proof}

Note that the exponent of $\hat{\beta}$ in Proposition \ref{e:flat-bound} is much better than the $2s$ that is needed.  This does not mean that Theorem \ref{th:necessity} can be improved as it is the non-flat balls that present the bottleneck.  Indeed, even in Hilbert spaces, where this flat ball argument was first presented, one gets the better bound $\sum_{B \in \cG_2} \beta_\Gamma(B) \diam(B) \leq C \ell(\Gamma)$ (see Corollary 3.26 of \cite{Schul-TSP}).

\section{Sufficiency}

We will prove the following theorem which will take care of the sufficient direction.

\begin{theorem} \label{th:sufficiency}
  Let $G$ be a step $s$ Carnot group with Hausdorff dimension $Q$ and $E \subset G$ a subset.  Define
  \begin{align*}
    \hat{\beta}(E) = \int_0^\infty \int_G \hat{\beta}_E(x,r)^{2s} ~dx \frac{dr}{r^Q}.
  \end{align*}
  There exists some constant $C > 0$ depending on $G$ and its Euclidean structure $|\cdot|$  so that if $\diam(E) + \hat{\beta}(E)$ is finite then $E$ lies on a rectifiable curve of length no more than $C(\diam(E) + \hat{\beta}(E))$.
\end{theorem}

One sees that Theorems \ref{th:necessity} and \ref{th:sufficiency} together imply Theorem \ref{th:main}.

The following lemma is an improvement on Lemma \ref{l:beta-balls}.  Note that it does not follow from Lemma \ref{l:beta-balls} because the closest point in $L$ to $p$ is chosen after the level $i$ is fixed.   As with Lemma \ref{l:beta-balls}, it holds for both $d_\infty$ and $d_{HS}$.

\begin{lemma} \label{l:euc-ball-contain}
  There exists a constant $C > 0$ depending only on $d$, $\eta_0 < 1$, and the Euclidean structure $|\cdot|$ so that if $p \in G$ and $L \subset G$ is a horizontal line for which $\max_i (d(\pi_i(p),\pi_i(L))/\ell)^{2i} = \eta < \eta_0$ for some $\ell > 0$, then $p \in L \cdot \delta_\ell(B_{\R^n}(C\eta^{1/2}))$.
\end{lemma}

\begin{proof}
  As usual, we dilate so that $\ell = 1$.  Choose $C_0 > 0$ so that $N(\pi_k(g)) \geq C_0 \sum_{1 \leq i \leq k} |g_i|^{1/i}$ for all $k$.  By choosing $\eta_0$ sufficiently small, we may suppose $\eta^{1/2s} < \min\{1,C_0/2\}$.

  We will prove for any $k$ that there exists $\eta_0 < 1$ and $c_1,...,c_k > 0$ so that if
  $$\max_{1 \leq i \leq k} d(\pi_i(p),\pi_i(L))^{2i} \leq \eta < \eta_0,$$
  then
  \begin{align}
    p \in L \cdot \left( \prod_{i=1}^k [-c_i\eta^{1/2},c_i\eta^{1/2}]^{n_i} \times \prod_{i=k+1}^s \R^{n_i}\right). \label{e:induct-subproof}
  \end{align}
  The lemma follows from this statement for $k = s$.

  The case of $k = 1$ is obvious.  Assume we have the statement up to $k-1$ and now assume $\max_{1 \leq i \leq k} d(\pi_i(p),\pi_i(L))^{2i} \leq \eta$.  Then by the induction hypothesis, we have
  $$p \in L \cdot \left( \prod_{i=1}^{k-1} [-c_i\eta^{1/2},c_i\eta^{1/2}]^{n_i} \times \prod_{i=k}^s \R^{n_i}\right).$$
  We choose $\eta_0$ so that $c_i\eta_0^{1/2} < 1/2$ for all $i \in \{1,...,k-1\}$.  Seeking a contradiction, let us suppose that
  $$p \in L \cdot \left( \prod_{i=1}^{k-1} [-c_i\eta^{1/2},c_i\eta^{1/2}]^{n_i} \times \left([-C\eta^{1/2},C\eta^{1/2}]^{n_i}\right)^c \times \prod_{i=k+1}^s \R^{n_i}\right),$$
  for some sufficiently large $C$ to be chosen (that will not depend on $p$).  We translate the setting so that $L = \{(vt,0,...,0)\}_{t \in \R}$ where $v \in S^{n_1-1}$ and $p = (x_1,...,x_s)$ where $|x_i| \leq c_i \eta^{1/2} < 1/2$ for $1 \leq i \leq k-1$ and $|x_k| \geq C\eta^{1/2}$.  Define
  \begin{align*}
    f(t) = d(\pi_k(p),\pi_k(vt,0,...,0)) \geq C_0\left(|x_1 - vt| + |x_k + P_k(-vt,0,...,0,x_1,...,x_{k-1})|^{1/k}\right),
  \end{align*}
  where we used the definition of $C_0$.  We will prove $f(t) > \eta^{1/2k}$ for all $t$, which will be our contradiction.

  First suppose $|t| \geq 1$.  Recall that $|v| = 1$ so that
  $$C_0|x_1 - vt| \geq (1 - |x_1|)C_0 \geq C_0/2 > \eta^{1/2k}.$$
  Thus, assume $|t| \leq 1$.  Then by Lemma \ref{l:BCH-bound}, we get that $|P_k(-vt,0,...,0,x_1,...,x_{k-1})| \leq C_1\eta^{1/2}$ for some $C_1$ depending only on $G$.  Thus, if we choose $C > C_1 + C_0^{-k}$, then
  \begin{align*}
    C_0|x_k + P_k(-vt,0,...,0,x_1,...,x_{k-1})|^{1/k} \geq C_0(C\eta^{1/2} - |P_k(-vt,0,...,0,x_1,...,x_{k-1})|)^{1/k} > \eta^{1/2k}.
  \end{align*}
\end{proof}

Following \cite{FFP}, we say that an ordered triple of points $p_1,p_2,p_3$ in a metric space is {\it orderable} if
\begin{align*}
  d(p_1,p_3) \geq \max\{d(p_1,p_2), d(p_2,p_3) \}.
\end{align*}
Note that $t_1,t_2,t_3 \in \R$ is orderable if and only if $t_2$ is between $t_1$ and $t_3$.

The next lemma says that the triangle inequality excess of orderable points in $\R^n$ can be bounded from above by $\frac{h^2}{\ell}$ where $h$ is the height of the triangle and $\ell$ is the diameter.

\begin{lemma} \label{l:euc-pythag}
  Let $\alpha > 0$ and $P_1,P_2,P_3 \in \R^n$ be orderable points such that $\alpha \ell \leq |P_i - P_j| \leq \ell$ for some $\ell > 0$.  Let $L$ be any affine line.  There exists a constant $C > 0$ depending only on $\alpha$ so that
  \begin{align}
    |P_1 - P_2| + |P_2 - P_3| - |P_1 - P_3| \leq C \max_{1 \leq i \leq 3} \frac{|P_i - L|^2}{\ell}. \label{e:euclidean-delta}
  \end{align}
\end{lemma}

\begin{proof}
  The three points lie on a Euclidean plane.  As the right hand side of \eqref{e:euclidean-delta} can only increase for lines not in the plane, we may assume $\R^n = \R^2$.  We rotate and dilate so that $\ell = 1$, $P_1 = 0$, and $P_3 = (\gamma,0)$ where $\gamma \in [\alpha, 1]$.  Let $P_2 = (x,y)$.  As $P_2 \notin \{P_1, P_3\}$, we must have $x \in (0,\gamma)$ as otherwise the orderability condition is violated.

  Let $L'$ be the affine line going through $P_1,P_3$.  We first claim that
  \begin{align*}
    h := |P_2 - L'| \leq 2\max_{1 \leq i \leq 3} |P_i - L|.
  \end{align*}
  Indeed, if $L$ is within $h/2$ of $P_1,P_3$ then $|P_2 - L| \geq h/2$, which gives the claim.  Otherwise, $L$ is at least $h/2$ away from one of $P_1$,$P_3$, which again gives the claim.  Thus, it now suffices to prove that
  \begin{align}
    |P_1 - P_2| + |P_2 - P_3| - |P_1 - P_3| \leq C |P_2 - L'|^2. \label{e:mod-euclidean-delta}
  \end{align}

  As the left hand side of \eqref{e:mod-euclidean-delta} is bounded by $2$, if $|P_2 - L'| >\epsilon_0$ for some $\epsilon_0 > 0$ sufficiently small to be determined, then the lemma immediately follows.  Thus, we may assume $|y| = h \leq \epsilon_0$.

  By choosing $\epsilon_0 < \alpha/2$ we get that $x \in (\alpha/2, \gamma-\alpha/2)$.  Note then that $P_2 = (x,y)$.  We have that
  \begin{align*}
    |P_1 - P_2| + |P_2 - P_3| - |P_1 - P_3| &= \sqrt{x^2 + h^2} + \sqrt{(\gamma - x)^2 + h^2} - x - (\gamma - x) \\
    &= \left( \frac{1}{\sqrt{x^2 + h^2} + x} + \frac{1}{\sqrt{(\gamma - x)^2 + h^2} + \gamma - x} \right) h^2 \\
    &\leq \left( \frac{1}{2x} + \frac{1}{2(\gamma - x)} \right) h^2 \\
    &\leq \frac{2}{\alpha} h^2,
  \end{align*}
  where we used the fact that $x \in (\alpha/2,\gamma-\alpha/2)$ in the last inequality.
\end{proof}

The next lemma says that in a metric space, the ordering of points is the same as the ordering of sufficiently close points on a subset isometric to $\R$.

\begin{lemma} \label{l:metric-order}
  Let $L \subseteq X$ be a subset isometric to $\R$ and $p_1,p_2,p_3 \in X$ be points so that $d(p_i,p_j) \geq \alpha \ell$ for some $\ell > 0$ and $p_i \in B(t_i, \alpha \ell/8)$ for points $t_i \in L$.  Then $p_1,p_2,p_3$ are orderable if and only if $t_1, t_2, t_3$ are orderable.
\end{lemma}

\begin{proof}
  By passing to the rescaled metric $\frac{1}{\ell} d$ (which preserves orderability in the original metric), we may assume $\ell = 1$.
  By a triangle inequality argument, we have that
  \begin{align*}
    d(t_i, t_j) \geq d(p_i, p_j) - \frac{\alpha}{4} \geq \frac{3}{4} \alpha.
  \end{align*}

  First assume $t_1,t_2,t_3$ are orderable.  Then
  \begin{align*}
    d(t_1,t_3) \geq \max\{d(t_1,t_2), d(t_2,t_3)\} + \frac{3}{4} \alpha.
  \end{align*}
  Here, we used the fact that $t_1,t_2,t_3$ lie in a subset isometric to $\R$.
  We have again by triangle inequality that
  \begin{align*}
    d(p_1, p_3) &\geq d(t_1,t_3) - \alpha/4, \\
    d(t_1,t_2) & \geq d(p_1, p_2) - \alpha/4, \\
    d(t_2,t_3) & \geq d(p_2, p_3) - \alpha/4.
  \end{align*}
  Thus,
  \begin{align*}
    d(p_1,p_3) \geq \max\{d(t_1,t_2), d(t_2,t_3)\} + \frac{\alpha}{2} \geq \max\{d(p_1,p_2), d(p_2,p_3)\}.
  \end{align*}

  Now assume $p_1,p_2,p_3$ are orderable.  As
  \begin{align*}
    d(t_1, t_3) &\geq d(p_1,p_3) - \alpha/4, \\
    d(p_1,p_2) & \geq d(t_1, t_2) - \alpha/4, \\
    d(p_2,p_3) & \geq d(t_2, t_3) - \alpha/4,
  \end{align*}
  we get that $d(t_1,t_3) \geq \max\{ d(t_1,t_2), d(t_2,t_3)\} - \alpha/2$.  As $d(t_i,t_j) \geq 3\alpha/4$, we must have that $t_2$ is between $t_1$ and $t_3$.
\end{proof}

\begin{lemma} \label{l:proj-order}
  Let $\alpha > 0$ and $p_1,p_2,p_3 \in G$ be orderable points so that $\alpha \ell \leq d_{HS}(p_i, p_j) \leq \ell$ for some $\ell > 0$ whenever $i \neq j$.  Let $L$ be any horizontal line.  If
  $$\max_{i \in \{1,2,3\}} \frac{d_{HS}(p_i,L)}{\ell} < \frac{\alpha}{8},$$
  then $\frac{\alpha}{2} \ell \leq |\pi_1(p_i) - \pi_1(p_j)| \leq \ell$ whenever $i \neq j$ and  $\pi_1(p_1), \pi_1(p_2), \pi_1(p_3)$ are also orderable in $\R^{n_1}$.
\end{lemma}

\begin{proof}
  As usual, we dilate so that $\ell = 1$.  We translate so that $L = \{(tv, 0,...,0) : t \in \R\}$ for some $v \in S^{n_1-1}$.  Then by assumption, there exists $t_1,t_2,t_3 \in \R$ so that $p_i \in (t_iv,0,...,0) \cdot B_G(\alpha/8)$.  We have by a triangle inequality argument that
  \begin{align*}
    |t_i - t_j| \geq d_{HS}(p_i, p_j) - \frac{\alpha}{4} \geq \frac{3}{4} \alpha
  \end{align*}
  whenever $i \neq j$.  Here, we used the fact that $N_{HS}(tv,0,...,0) = |t|$.

  We first prove that $\frac{\alpha}{2} \leq |\pi_1(p_i) - \pi_1(p_j)| \leq 1$.  The upper bound follows from the fact that $\pi_1$ is 1-Lipschitz.  Thus, it suffices to prove the lower bound.  We have by 1-Lipschitzness of $\pi_1$ that
  \begin{align*}
    \max_{i \in \{1,2,3\}} d_{HS}(\pi_1(p_i),\pi_1(t_iv,0,...,0)) < \frac{\alpha}{8},
  \end{align*}
  and so a triangle inequality argument gives
  \begin{align*}
    |\pi_1(p_i) - \pi_1(p_j)| \geq |t_iv - t_jv| - \frac{\alpha}{4} \geq \frac{\alpha}{2}.
  \end{align*}

  We now prove orderability.  As $L$ is isometric to $\R$, we have by Lemma \ref{l:metric-order} that $t_1, t_2, t_3$ are orderable.  As $\pi_1$ is 1-Lipschitz and an isometry when restricted on $L$, we get by another application of Lemma \ref{l:metric-order} that $\pi_1(p_1), \pi_1(p_2), \pi_1(p_3)$ is also orderable.
\end{proof}

The following lemma will allow us to Taylor expand the Hebisch-Sikora norm.
\begin{lemma} \label{l:HS-taylor}
  For any Carnot group $G$ and $\alpha \in (0,1)$, there exists a constant $C > 0$ so that if $\alpha \leq N_{HS}(x_1,...,x_{s-1},0) \leq 1$ and $|y| \leq 1/C$, then
  \begin{align*}
    0 \leq N_{HS}(x_1,...,x_{s-1},y) - N_{HS}(x_1,...,x_{s-1},0) \leq C|y|^2.
  \end{align*}
\end{lemma}

\begin{proof}
  As balls of the Hebisch-Sikora norm centered at the origin are axis-aligned ellipsoids, the left hand inequality is obvious.
  Define $K = \{(x_1,...,x_{s-1},0) : \alpha \leq N_{HS}(x_1,...,x_{s-1},0) \leq 1\}$, which is a compact set.
  Recall that $N_{HS}$ is smooth on $G \setminus \{0\}$.  For $x = (x_1,...,x_{s-1})$, we define the function
  \begin{align*}
    f_x(y) = N_{HS}(x_1,...,x_{s-1},y) - N_{HS}(x_1,...,x_{s-1},0).
  \end{align*}
  We have already proven that $f_x(y) \geq 0$ always.  Furthermore, $f_x(0) = 0$ and so $Df_x(0) = 0$.  Thus, by smoothness of $N_{HS}$, we have that $f_x(y) \leq \|D^2f_x(0)\| |y|^2$ when $|y| \leq \epsilon_x$ where $\|D^2f_x(0)\|$ is the operator norm of the Hessian of $f_x$ at 0.  Here, $\epsilon_x$ is positive for all $x \in K$ and can be chosen in a continuous way.  By compactness of $K$, we get the lemma by choosing
  \begin{align*}
    C = \max \left\{\max_{x \in K} \epsilon_x^{-1}, \max_{x \in K} \|D^2f_x(0)\| \right\}.
  \end{align*}
\end{proof}

The following proposition is the main result of this section.

\begin{proposition}
  For any $\alpha \in (0,1)$ there exists $C > 0$ depending on $d_{HS}$ and the Euclidean structure $|\cdot|$ so that if $p_1,p_2,p_3 \in E$ are orderable points such that $\alpha \ell \leq d_{HS}(p_i,p_j) \leq \ell$ for $i \neq j$ for some $\ell > 0$, then
  \begin{align*}
    d_{HS}(p_1,p_2) + d_{HS}(p_2,p_3) - d_{HS}(p_1,p_3) \leq C \max_{1 \leq i \leq 3} \max_{1 \leq k \leq s} \frac{d_{HS}(\pi_k(p_i),\pi_k(L))^{2k}}{\ell^{2k-1}}
  \end{align*}
  where $L$ is any horizontal line.
\end{proposition}

\begin{proof}
  We dilate and translate the setting so that $\ell = 1$ and $L = \{(vt,0,...,0)\}_{t \in \R}$ where $v \in S^{n_1-1}$.  Let $\mu > 0$ be a constant to be determined.  Note that the left hand side of the desired inequality is bounded from above by 2.  If the non-constant term of the right hand side is bounded from below by  $\mu$, then the inequality holds with $C = 2/\mu$.  Thus, we may suppose that 
  \begin{align}
    \eta := \max_{1 \leq i \leq 3} \max_{1 \leq k \leq s} d_{HS}(\pi_k(p_i),\pi_k(L))^{2k} \leq \mu \label{e:eta-defn}
  \end{align}
  for some $\mu$ to be determined.  Our goal then is to bound the left hand side of the desired inequality by some multiple of $\eta$.

  Let $P_i = \pi_1(p_i)$.  We will first let $\mu \leq (\alpha/8)^{2s}$ so that we can apply Lemma \ref{l:proj-order} to get that $\alpha/2 \leq |P_i - P_j| \leq 1$ whenever $i \neq j$ and $P_1, P_2,P_3$ are orderable in $\R^{n_1}$.

  We then let $\mu < \eta_0$ of Lemma \ref{l:euc-ball-contain} to get that there exists some $C_0 > 0$ so that $p_i = (t_iv,0,...,0) \cdot u_i$ where $u_i \in B_{\R^n}(C_0\eta^{1/2})$ and $t_1,t_2,t_3 \in \R$.  Thus, $p_i^{-1}p_j = u_i^{-1} \cdot ((t_j-t_i)v,0,...,0) \cdot u_j$.  Note that $P_i = t_iv + \pi_1(u_i) \in \R^{n_1}$.  By repeated applications of Lemma \ref{l:BCH-bound}, we get that
  \begin{align*}
    p_i^{-1}p_j = (P_j - P_i, q_{i,j}^{(2)},...,q_{i,j}^{(s)})
  \end{align*}
  where $\max_{2 \leq k \leq s} |q_{i,j}^{(k)}| \leq C_1\eta^{1/2}$ for some $C_1 > 0$.  

  Let $(x_1,...,x_s) = p_1^{-1}p_2$.  As $|P_1 - P_2| \geq \alpha/2$, we have that $\alpha/2 \leq N_{HS}(x_1,...,x_{s-1},0) \leq 1$.  Let $C$ be the constant from Lemma \ref{l:HS-taylor}.  Recall that
  $$|x_s| = |q^{(s)}_{1,2}| \leq C_1\eta^{1/2} \leq C_1 \mu^{1/2}.$$
  Thus, if we take $\mu$ sufficiently small, we get that $|x_s| < 1/C$.  We can then apply Lemma \ref{l:HS-taylor} to get
  \begin{align*}
    N_{HS}(x_1,...,x_s) - N_{HS}(x_1,...,x_{s-1},0) \leq C |x_s|^2 \leq CC_1^2 \eta.
  \end{align*}
  We repeat this process of bounding
  $$d_{HS}(\pi_i(p_1), \pi_i(p_2)) - d_{HS}(\pi_{i-1}(p_1), \pi_{i-1}(p_2)) = N_{HS}(x_1,...,x_i) - N_{HS}(x_1,...,x_{i-1},0)$$
  in $\pi_i(G)$ by some multiple of $\eta$ and then telescope the resulting bounds to get
  \begin{align*}
    d_{HS}(p_1,p_2) = d_{HS}(\pi_s(p_1),\pi_s(p_2)) \leq d_{HS}(\pi_1(p_1), \pi_1(p_2)) + C_2 \eta = |P_1 - P_2| + C_2 \eta
  \end{align*}
  for some $C_2 > 0$.  Likewise,
  \begin{align*}
    d_{HS}(p_2,p_3) &\leq |P_2 - P_3| + C_2 \eta.
  \end{align*}
  Also, we get from repeated applications of the other inequality of Lemma \ref{l:HS-taylor} that
  \begin{align*}
    d_{HS}(p_1,p_3) &\geq |P_1 - P_3|.
  \end{align*}
  Altogether, we get that
  \begin{align*}
    d_{HS}(p_1,p_2) + d_{HS}(p_2,p_3) - d_{HS}(p_1,p_3) \leq |P_1 - P_2| + |P_2 - P_3| - |P_1 - P_3| + 2C_2 \eta.
  \end{align*}

  Now applying Lemma \ref{l:euc-pythag} gives us that $|P_1 - P_2| + |P_2 - P_3| - |P_1 - P_3| \leq C_3 \eta$ for some $C_3$ depending only on $\alpha$.  Recalling the definition of $\eta$ in \eqref{e:eta-defn}, we get the proposition.
\end{proof}

We immediately get the following corollary.

\begin{corollary}\label{c:our-cor}
  For any $\alpha \in (0,1)$ there exists $C > 0$ depending on $\alpha$, $d_{HS}$, and the Euclidean structure so that the following holds.  If $p_1,p_2,p_3 \in E$ are orderable points such that $\alpha \ell \leq d_{HS}(p_i,p_j) \leq \ell$ for some $\ell > 0$ when $i \neq j$, then
  \begin{align*}
    d_{HS}(p_1,p_2) + d_{HS}(p_2,p_3) - d_{HS}(p_1,p_3) \leq C \hat{\beta}_E(B(p_i,2\alpha\ell))^{2s} \ell
  \end{align*}
  for any $i \in \{1,2,3\}$.
\end{corollary}

\begin{proof}[Proof of Theorem \ref{th:sufficiency}]
  We may assume $E$ is bounded as otherwise there is nothing to prove.  To prove the theorem, we will use the farthest insertion algorithm of section 3 of \cite{FFP}, which we assume the reader is familiar with.  As a refresher, their construction starts with $\{\Delta_k\}_{k=0}^\infty$, a nested set of finite maximal $2^{-k}$-nets of $E$.  They then build a sequence of connected piecewise geodesic curves $\{\Gamma_k\}_{k=0}^\infty$ where $\Delta_k$ is contained at the endpoints of the geodesic segments (see the bottom of p. 467 of \cite{FFP}) and pass to a limit to get $\Gamma$.

  The construction of \cite{FFP} is in the Heisenberg group and uses the usual $\beta$ (which they call $\beta_\H$).  However, they make no reference to the Heisenberg group except through their Proposition 3.2 and Corollary 2.16 (there is a typo as a square is missing in their $\beta_\H(P_i,C_0t)$).  By their Remark 3.1, their Proposition 3.2 holds for all Carnot groups.  As for their Corollary 2.16, we will replace that with our Corollary \ref{c:our-cor}.  Additionally, in Case A, we use the following bound instead of their (3.3):
  \begin{align*}
    \ell(\Gamma_j) - \ell(\Gamma_{j-1}) \leq C2^{-j} \leq \frac{C}{\epsilon_0^{2s}} 2^{-j} \beta_E(P,C_12^{-j})^{2s} \leq \frac{C}{\epsilon_0^{2s}} 2^{-j} \hat{\beta}_E(P,C_12^{-j})^{2s},
  \end{align*}
  This, along Corollary \ref{c:our-cor}, replaces the presence of $\beta_\H^2$ in their upper bound with $\hat{\beta}^{2s}$.

  One important difference is that \cite{FFP} uses the Carnot-Carath\'eodory metric, but we use $d_{HS}$.  Thus, all nets, balls, etc. will be done with respect to $d_{HS}$.  As $d_{HS}$ is not geodesic, we must modify the method of \cite{FFP}.  We will actually apply the algorithm of \cite{FFP} to inductively build {\it abstract} connected graphs $\tilde{\Gamma}_k$ whose vertices contain $\Delta_k$.  The ``length'' of an abstract graph $\ell(\tilde{\Gamma}_k)$ is the sum of the $d_{HS}$ length of all abstract edges in the graph, and all references to geodesic segments in the algorithm will be replaced by abstract edges.

  The results of \cite{FFP} show that their $\ell(\Gamma_k)$ is bounded their $\beta_\H(E) + \diam(E)$ where $\beta_\H(E)$ is the Carleson integral of $\beta_\H^2$.  Indeed, let $\mathcal{A}_k$ denote all geodesic segments of $\mathcal{A}$ (as defined right before Case B2(i)) that appear in $\Gamma_k$.  Then, using the notation of \cite{FFP}, we can conclude using their equations (3.3), (3.4), (3.10), (3.11) (there is a typo in (3.11) as it should be $\tau_2([Q_1,R_0])$ on the left hand side) that
  \begin{align*}
    \ell(\Gamma_k) &= \sum_{j=0}^{k-1} (\ell(\Gamma_{k+1}) - \ell(\Gamma_k)) + \ell(\Gamma_0) \\
    &\leq C\sum_{j \in \N} \sum_{P \in \Delta_j} \beta_\H(P,C_3 2^{-j})^2 2^{-j} + \sum_{I \in \mathcal{A}_k} \tau_1(I) + \sum_{I \in \mathcal{A}_k} \tau_2(I) + C \diam(E) \\
    &\leq C\sum_{j \in \N} \sum_{P \in \Delta_j} \beta_\H(P,C_3 2^{-j})^2 2^{-j} + \frac{1}{10^6} \ell(\Gamma_k) + \frac{1}{10^3} \ell(\Gamma_k) + C \diam(E),
  \end{align*}
  which implies that $\ell(\Gamma_k) \leq C(\beta_\H(E) + \diam(E))$ for all $k$ as the sum of the $\beta_\H^2$ is just a discretized Carleson integral.  This is similar to their estimate at the top of p. 468.  With the modifications we outlined above, we would get the following uniform bound for the abstract graphs
  \begin{align*}
    \ell(\tilde{\Gamma}_k) \leq C (\diam(E) + \hat{\beta}(E)).
  \end{align*}

  Once the abstract graphs are built, we replace each abstract edge in $\tilde{\Gamma}_k$ with quasi-convex curves using \eqref{eq:qc} to get a sequence of actual continua $\{\Gamma_k\}_{k=0}^\infty$ such that $\Delta_k \subset \Gamma_k$.  Note that the lengths are still uniformly bounded 
  $$\Hd^1(\Gamma_k) \leq C \ell(\tilde{\Gamma}_k) \leq C(\diam(E) + \hat{\beta}(E)).$$

  Finally, as each $\Delta_k$ is finite and so compact and lies in $\Gamma_j$ for $j \geq k$, we get from their Theorem 5.1 and a diagonalization argument that the limiting curve $\Gamma$ of (a subsequence of) the $\Gamma_k$ contains each $\Delta_k$ and $\Hd^1(\Gamma) \leq C(\diam(E)+\hat{\beta}(E))$.  As $\Gamma$ is compact and $\overline{\bigcup_{k=1}^\infty \Delta_k} \supseteq E$, we then get that $E \subseteq \Gamma$.  This proves the theorem.
\end{proof}

\section{Proof of Proposition \ref{p:euc-char}} \label{s:proof}
We first prove the equivalence of (1) and (2).
We prove $(2) \leq (1)$ (up to constants) first.  Let $L$ be a horizontal line so that
\begin{align*}
  \sum_{i=1}^s \sup_{z \in E \cap B(x,r)} \left( \frac{d_{HS}(\pi_i(z), \pi_i(L))}{r} \right)^{2i} \leq 2 \hat{\beta}_E(x,r)^{2s}.
\end{align*}
Then we also have
\begin{align*}
  \sup_{z \in E \cap B(x,r)} \max_{i \in \{1,...,s\}} \left( \frac{d_{HS}(\pi_i(z), \pi_i(L))}{r} \right)^{2i} \leq 2 \hat{\beta}_E(x,r)^{2s}.
\end{align*}
and so Lemma \ref{l:euc-ball-contain} gives that
\begin{align*}
  E \cap B_G(x,r) \subseteq L \cdot \delta_r(B_{\R^n}(C \sqrt{2} \hat{\beta}_E(x,r)^s))
\end{align*}
which gives the needed bound.

We now prove the converse bound.  Suppose that
\begin{align*}
  E \cap B_G(x,r) \subseteq L \cdot \delta_r(B_{\R^n}(\epsilon^s)).
\end{align*}
Note that we can take $\epsilon \leq M$ where $M$ is a constant depending only on $d$ and $|\cdot|$ so that $B_G(0,1) \subseteq B_{\R^n}(M^s)$.

Let $z \in E \cap B(x,r)$.  Then there is a $p \in L$ so that $z \in p \cdot \delta_r(B_{\R^n}(\epsilon^s))$. 
Then Lemma \ref{l:beta-balls} gives
\begin{align*}
  \left(\frac{d_{HS}(\pi_i(z),\pi_i(p))}{r}\right)^i \leq C\epsilon^s, \qquad \forall i \in \{1,...,s\}.
\end{align*}
As $p \in L$ and $z \in E \cap B(x,r)$ was arbitrary, we get for each $i$ that
\begin{align*}
  \sup_{z \in E \cap B(x,r)} \left(\frac{d_{HS}(\pi_i(z),\pi_i(L))}{r}\right)^{2i} \leq C^2\epsilon^{2s}.
\end{align*}
which allows us to conclude that
\begin{align*}
  \hat{\beta}_E(x,r)^{2s} = \sum_{i=1}^s \sup_{z \in B(x,r) \cap E} \left(\frac{d_{HS}(\pi_i(z), \pi_i(L))}{r}\right)^{2i} \leq C^2s \epsilon^{2s}
\end{align*}
as needed.  This proves the equivalence between (1) and (2).


We now establish the equivalence between (2) and (3).  Note that the $\epsilon$ of (2) can be bounded by the same $M$ as above. There is also a similar upper bound for the $\epsilon$ of (3) which depends only on $d$ and $\rho$.  Now, by translating and dilating, we can rewrite (2) as
\begin{align*}
  \inf_L \inf \{ \epsilon > 0 : \delta_{1/r}(x^{-1}E) \cap B_G(0,1) \subset L \cdot B_{\R^n}(\epsilon^s) \}.
\end{align*}
The left-invariance of $\rho$ along with \eqref{e:riem-compare} gives that there exists a $C_0 > 1$ so that
\begin{align*}
  L \cdot B_{\R^n}(\epsilon^s/C_0) \subseteq \{z \in G : \rho(z,L) < \epsilon^s\} \subseteq L \cdot B_{\R^n}(C_0\epsilon^s).
\end{align*}
The equivalence of (2) and (3) now easily follows.
\qed

\end{document}